\documentclass{cmslatex}
\usepackage[paperwidth=7in, paperheight=10in, margin=.875in]{geometry}
 \usepackage[backref,colorlinks,linkcolor=red,anchorcolor=green,citecolor=blue]{hyperref}
\usepackage{amsfonts,amssymb}
\usepackage{amsmath}
\usepackage{graphicx}
\usepackage{cite}
\usepackage{enumerate}
\renewcommand{\theequation}{\arabic{section}.\arabic{equation}}

   \allowdisplaybreaks
\begin{document}
 \title{Qualitative Behavior of Solutions to a  Forced Nonlocal Thin-Film Equation\thanks{Received date, and accepted date (The correct dates will be entered by the editor).}}


\author{Jinhong~Zhao\thanks{School of Mathematics, Jilin University, Changchun, 
		Jilin Province 130012, China, (jhzhao23@mails.jlu.edu.cn).}
	\and Bin~Guo\thanks{School of Mathematics, Jilin University, Changchun, 
		Jilin Province 130012, China, (bguo@jlu.edu.cn).}}

         \pagestyle{myheadings} \markboth{Forced Nonlocal Thin-Film Behavior}{Jinhong~Zhao  and Bin~Guo} \maketitle

          \begin{abstract}
               We study a one-dimensional nonlocal degenerate fourth-order parabolic equation with inhomogeneous forces relevant to hydraulic fracture modeling. Employing a regularization scheme, modified energy/entropy methods, and novel differential inequality techniques, we establish global existence and long-time behavior results for weak solutions under both time-and space-dependent and time-and space-independent  inhomogeneous forces. Specifically, for the time-and space-dependent force $S(t, x)$, we prove that the  solution converges  to $\bar{u}_0+\frac{1}{|\Omega|}\int_0^\infty \int_\Omega S(r, x)\, dxdr $,  where $\bar{u}_0=\frac{1}{|\Omega|}\int_{\Omega}u_{0}(x)\,dx$ is the spatial average of the initial data, and we provide bilateral estimates for the convergence rate. 
               For the time-and space-independent force $S_0$, we show that the solution approaches the linear function $\bar{u}_0 + tS_0$ at an exponential rate.
          \end{abstract}


\section{Introduction}

In this paper, we study a nonlocal degenerate parabolic equation with inhomogeneous forces
\begin{equation}\label{1.1}
\begin{cases}
\partial_t u + \partial_x\left(u^n \partial_x I(u) \right) = S(t, x),~~&(t,x) \in  (0,T) \times \Omega,\\
\partial _x u =u^n \partial_x I(u) = 0,~~&(t,x) \in (0,T) \times \partial\Omega,\\
u(0, x)=u_{0}(x),~~&x \in \Omega,
\end{cases}
\end{equation}
where  	$ \Omega=(a,b)$ is a bounded interval on the real line, $ T > 0 $ is a given time, and  $n>0$ is the diffusion growth exponent.  The operator $I=-\left( -\Delta \right)^s $ is a nonlocal negative  elliptic  operator of order $2s$, with $s \in (0, 1)$, which is defined precisely in section 2.
$u(t,x)$ denotes the fracture aperture, and $S(t,x)$ represents the net volume increase rate of fluid per unit time per unit fracture length.
Problem \eqref{1.1} describes the dynamic evolution of the fracture aperture, where the propagation is driven by the pressure exerted by the viscous fluid that fills the fracture.

Hydraulic fracturing, a critical technique in the global energy landscape, enhances reservoir permeability by creating artificial fractures via high-pressure fluid injection into subsurface rock ($S(t,x)>0$). As cracks initiate and propagate during this process, studying their long-time behavior offers essential theoretical insights for technology optimization.

For $s=1$, $S(t,x)=0,$ the first equation in problem \eqref{1.1} becomes the  classical thin-film equation
\begin{align}\label{111}
\partial_t u + \partial_x\left(u^n  \partial_{xxx}^3 u\right) = 0,
\end{align}
whose study dates back to the work of Bernis and Friedman \cite{FBAF} in 1990. They first obtained the existence of non-negative weak solutions to problem \eqref{111} subject to  Neumann boundary conditions $u_x = u_{xxx} = 0$ in one dimension, and investigated the regularity of solutions as well as the monotonicity of supports. 
Moreover, Beretta, Bertsch, and Dal Passo \cite{BBDP} and Bertozzi and Pugh \cite{BP} studied problem \eqref{111} in the one-dimensional case, with a focus on the relationship between the qualitative behavior of solutions and the diffusion growth exponent $n$. 
%
Subsequently, Bernis \cite{B2,B3} showed that for $0 < n < 3$, solutions to \eqref{111} have the properties of finite speed of propagation and the waiting time phenomenon.
%
After these pioneer works, there are many research activities regarding different slip conditions on the fluid-solid interface.
For instance,  "strong slippage" ($n \in (1, 2)$), "weak slippage" ($n \in (2, 3)$), "no-slip condition"($n = 3$) \cite{AOS}, "Navier-slip condition" ($n = 2 $) \cite{WJA} and  Hele-Shaw flow ($n = 1$) \cite{LGF}.

Regarding existence results in higher dimensions, Dal Passo, Garcke, and Gr${\rm \ddot{u}}$n \cite{DPGG} improved the entropy estimate and extended the results of either \cite{BBDP} or \cite{BP} to the higher-dimensional case. 
Specifically, their results hold for exponents  $n > \frac{1}{8} $ in the two-dimensional case and  $\frac{1}{8} < n < 4$ in three dimensions. 
Subsequently, Li \cite{L1,L2}  established the existence, positivity, and further properties of solutions to problem \eqref{111} separately in four-dimensional space and in dimensions five and above.

For the convergence analysis, Carlen and Ulusoy \cite{EAS1,EAS2} derived an explicit decay rate in the $H^1(\Omega)$-norm for classical solutions to \eqref{111}, extending the earlier $L^1$ and $L^\infty$-norm results from \cite{BBDP, BP, JAG, LJS}.
Tudorascu \cite{TAL} further demonstrated  the exponential decay of weak solutions in $H^1(\Omega)$. 
More recently, Chugunova, Ruan and Taranets \cite{CMRY} first studied the long-time behavior of non-negative solutions to \eqref{111} driven by time-independent inhomogeneous forces $S(x)$ and $S_0$.
Additionally, Slep${\rm \check{c}}$ev \cite{DSL} considered the existence and qualitative properties of self-similar solutions.

Recent studies indicate that degenerate parabolic equations with nonlocal operators outperform classical thin-film models in hydraulic fracturing simulations, resulting in the following order $2s+2$ nonlocal equation
\begin{equation}\label{222}
\partial_t u + \partial_x\left(u^n \partial_x I(u) \right) = 0,
\end{equation}
where $I=-\left( -\Delta \right)^s$ and the parameter $s \in (0,1)$ is related to the properties of the medium in which the fracture spreads.
For  $s=0$,  \eqref{222} reduces to the Porous Medium Equation. When
$s = 1$, it coincides with the Thin-Film Equation. For the critical case
$s = \frac{1}{2}$, Imbert and Mellet \cite{CIAM} first established the existence of non-negative solutions to problem \eqref{222}. Tarhini \cite{RT} later extended this existence result to all $s \in (0,1)$.  Furthermore, De Nitti and Taranets \cite{NDRMT} employed a localized entropy estimate and a Stampacchia-type lemma to prove finite speed of propagation and derive sufficient conditions for the waiting-time phenomenon in one dimension.
Subsequently, De Nitti, Lisini, Segatti, and Taranets \cite{NDSL} generalized these results to bounded domains in  $\mathbb{R}^d$.

Motivated by aforementioned works, we aim to study the existence and long-time behavior of non-negative solutions to problem \eqref{1.1}. In particular, we aim to
discuss the dependence of the asymptotic behavior of solutions on inhomogeneous forces. 
To the best of our knowledge, such problem has not been rigorously studied previously.  
Indeed, in dealing with such problems, we have to face some challenges including the lack of $L^{1}$ mass conservation and  loss of the monotonicity of the energy functional.
Consequently, we develop novel techniques and employ energy/entropy methods to resolve these issues.
Notably, the main novelties of this work include:

$\bullet$ The nonlocal nature of the operator $I$ poses challenges for analyzing the dissipation  functional  $D(u)(t)=\int_{\Omega} |u|^{n} \left| \partial_{x} I(u)\right| ^{2}  dx$. Consequently, existing techniques are  not immediately available.
To address this challenge, by analyzing the sign of $I(u)$ at the critical  point of the solution $u$, we establish a new weighted interpolation inequality (Lemma \ref{2007lem}), which enables us to derive a relationship between the dissipation  functional  and the  energy functional  $J(u)(t)=\int_{\Omega} |(-\Delta)^{\frac{s}{2}} u|^2 \,dx$.


$\bullet$ Another pivotal step lies in the derivation -- via a novel approach -- of bilateral bounds for solutions to a non-autonomous differential inequality (Lemma \ref{inequ}). 
By further combining weighted interpolation inequality, energy estimates, and entropy estimates, we construct a differential inequality for the energy functional, thereby obtaining bilateral estimates  of the convergence rate. 
Specifically, as time tends to infinity, the solution converges to the average value (Theorem \ref{ex1})
$$\bar{u}_0+\frac{1}{|\Omega|}\int_0^\infty \int_\Omega S(r, x)\, dxdr.$$

$\bullet$ As a byproduct, we obtain the local $L^1$-in-time estimate for the dissipation functional $D(u)(t)$  as $t\to\infty$, and provide an explicit convergence rate.

Before stating our main results, we first give the definition of the weak solution to problem \eqref{1.1}.
\begin{definition}\label{def.eq2}
	\rm Let $n \ge 1$, $s \in (0,1)$, $ u_0(x) \in  H^s(\Omega)$ and $S(t, x) \in  L^1(0, \infty;H^s(\Omega) )$.
	We say that a function
	$$u(t,x) \in L^{\infty}(0,T;H^s(\Omega) ) \cap  L^2(0,T;H^{s+1}(\Omega) )$$
	is a solution  to problem \eqref{1.1} in $[0,T] \times \Omega$ if $u$ satisfies
	\begin{align}\label{ex1.defi}
	\iint_{\Omega_T} u \partial_t \varphi \,dx dt
	- \iint_{\Omega_T} nu^{n-1} \partial_x u I(u) &\partial_x \varphi \,dx dt
	- \iint_{\Omega_T} u^{n} I(u) \partial_{xx}^2 \varphi \,dx dt\nonumber\\
	&=  -\int_{\Omega}u_0 \varphi(0,\cdot)\,dx - \iint_{\Omega_T} S(t,x) \varphi \,dx dt,
	\end{align}
	for all $\varphi \in  \mathcal{D}([0, T ) \times \bar{\Omega} ) $ satisfying $\partial_x \varphi = 0$ on $(0, T) \times \partial \Omega$. Here, the definition of  $\mathcal{D}$ will be found in Subsection 2.4.
	Moreover, we say that $u$ is a global-in-time solution to problem \eqref{1.1} if $u$ is a solution to problem \eqref{1.1} in $[0, T'] \times \Omega$ for all $T' > 0$.
	
\end{definition}


For the simplicity of our statement, we define the entropy function $ G : \mathbb{R} \rightarrow [0, +\infty) $  satisfies for any $z > 0 $,
\begin{align*}
G(z) = \int_A^z \int_A^r \frac{1}{t^n} \, dt \, dr,
\end{align*}
where $A>0$, and  $G(z) = +\infty $ for any $ z < 0 $, $ G(0) = \lim\limits_{z \to 0^+} G(z) $. 
Next, for the time-and space-dependent forces $S(t,x)$, the main results concerning the global existence and long-time behavior of weak solutions to problem \eqref{1.1} are stated as follows.

\begin{theorem}\label{ex1}
	{\rm (Time-and space-dependent forces $S(t, x)$)}
	Let $n \ge 1$, $s \in (\frac{1}{2},1)$.  Assume that $0 \le u_0(x) \in  H^s(\Omega)$ and $0 \le S(t, x) \in  L^1(0, \infty;H^s(\Omega) ) \cap  L^1(0, \infty;L^1(\Omega) )$ satisfying
	\begin{equation*}
	\int_{\Omega}G(u_0)\,dx < \infty.
	\end{equation*}
	Then for every fixed $T>0$, problem \eqref{1.1} has a global non-negative weak solution $u$   in the sense of  Definition \ref{def.eq2}.
	
	Meanwhile,  $u$ satisfies for almost every $t\in (0, T)$,
	\begin{equation}\label{ex1.mass}
	\int_{\Omega} u(t,x) \,dx = \int_{\Omega} u_0(x) \,dx + \int_0^t \int_{\Omega} S(r, x)  \,dxdr,
	\end{equation}
	and
	\begin{equation}\label{ex1.energy}
	[u]^2_{s,\Omega}+ 2\int_{0}^{T}\int_{\Omega} g^2 \,dx dt
	\le [u_0]^2_{s,\Omega} - 2\int_{0}^{T}\int_{\Omega} S(t, x)I(u) \,dx dt,
	\end{equation}
	where the function $g \in L^2(\Omega_T)$ satisfies $g= \partial_x( u^{\frac{n}{2}} I(u)) - \frac{n}{2} u^{\frac{n-2}{2}} \partial_x u I(u)$ in $\mathcal{D}'(\Omega)$, and
	\begin{equation}\label{ex1.entropy}
	\int_{\Omega} G(u)\, dx + \int_{0}^{t} [u]^2_{s+1,\Omega,N}\, dr
	\le \int_{\Omega} G(u_0)\, dx.
	\end{equation}

	Moreover, there exists a constant	$C > 0$ such that  the global solution $u$ has the following asymptotic behavior	
	\begin{align}\label{ex3.uK1}
	[u_0]_{s,\Omega}  \left( 1+ C [u_0]_{s,\Omega}^2t \right)^{-\frac{1}{2}}
	\le 	&\left\| u - \bar{u}_0 - \frac{1}{|\Omega|} \int_{0}^\infty \int_{\Omega}  S(r,x) \, dx dr\right\|_{H^s(\Omega)} \nonumber\\
	\le  & \; C M_0\left( 1+C(1-\kappa)M_0^2 t \right) ^{-\frac{1}{2}}
	+ C \int_{\kappa t}^t [S(r,x)]_{s,\Omega}\, dr \notag\\
	&+ |\Omega|^{-\frac{1}{2}}\int_{t}^\infty \int_{\Omega}  S(r,x) \, dx  dr ,    ~~~\forall t>0,
	\end{align}
	where  $\kappa\in(0,1),$ $0<M_0:=[u_0]_{s,\Omega} + \int_0^\infty [S(r,x)]_{s,\Omega}\, dr <+\infty.$ 
\end{theorem}

\begin{remark}
{\rm  Let $S(t,x) = (1+t)^{-\beta}a(x)$ with $\beta>1$, where the condition $\beta >1$ ensures the time-integrability of $S(t,x)$. According to \eqref{ex3.uK1}, when $1<\beta < \frac{3}{2}$, the convergence rate of the solution is governed by the inhomogeneous force; when $\beta \ge  \frac{3}{2}$,  the convergence rate is dominated by the dissipative term instead.
In the latter case, the solution converges to $\bar{u}_0+ \frac{1}{|\Omega|} \int_{0}^\infty \int_{\Omega} S(r,x) \,dx dr$ in the $H^s(\Omega)$-norm at the polynomial rate $O(t^{-1/2})$,    
and we show that this convergence rate is almost optimal.}
\end{remark}


\begin{remark}
	{\rm For the time-independent but space-dependent force $S(x)$, with the aid of Lemma \ref{inequ1}, we prove that the difference between the weak solution and the linear function $\bar{u}_0 + \frac{t}{|\Omega|}\int_\Omega S(x)\, dx$ remains uniformly bounded in $H^s (\Omega )$,  that is, there exists a constant $C > 0$  such that for all $t  >0$,
		\begin{align}\label{uK}
		&\left\| u -  \bar{u}_0  -  \frac{t}{|\Omega|} \int_\Omega S(x)\, dx\right\| _{{H}^s(\Omega)} \nonumber\\
		\le & \begin{cases}
		CR, & \text{if } [u_0]_{s,\Omega} \le R,\\
		CR+C\left( \left([u_0]_{s,\Omega} -R \right) ^{-2}+\frac{1}{2|\Omega|^2H_0}t \right)^{-\frac{1}{2}}, & \text{if }[u_0]_{s,\Omega} > R,
		\end{cases}
		\end{align}
		where $H_0=\int_{\Omega}G(u_0)\,dx <\infty$ and $R=\left(4|\Omega|^2H_0 [S(x)]_{s,\Omega}  \right) ^{\frac{1}{3}}$.
		
		By comparing \eqref{ex3.uK1} and \eqref{uK}, it is readily apparent that, owing to the time-integrability of the force, one can obtain sharper decay estimates for the solution.  }
\end{remark}

Finally, when the inhomogeneous force is the time-and space-independent force $S_0$, we may obtain an explicit decay rate. Namely,
\begin{theorem}\label{ex2}
	{\rm(}Time-and space-independent forces  $S_0${\rm)}
	Let $S(x) = S_0 \ge 0$, and $ u $ be a weak solution from Theorem \ref{ex1}. Then there exists constants $C > 0$ such that
	$$\| u-\bar{u}_0-tS_0 \|_{H^s(\Omega)}
	\le C \| u_{0} \|_{{H}^s(\Omega)} e^{-C\int_{T_0}^{t} \left(\bar{u}_0  - C(\Omega)[u_{0}]_{s,\Omega} +  S_0r \right) ^n  \, dr},~~~\forall t \geq T_0,$$
	where  $T_0= \frac{\left[C(\Omega)[u_{0}]_{s,\Omega}- \bar{u}_0  \right] _+}{S_0}$.
\end{theorem}

{\bf Outline of the paper. }
The structure of our paper is as follows:
In Section 2, we introduce some preliminary knowledge.  
Section 3 presents several key lemmas for analyzing the long-time behavior of solutions.  
In Sections  4 and 5, we obtain the global existence and long-time behavior of non-negative weak solutions to problem \eqref{1.1} for two distinct types of inhomogeneous forces, respectively.
Section 6 is devoted to concluding comments and future research directions.

\section{Preliminaries}
In this section,  we will give the related concepts of fractional Sobolev spaces and fractional Laplacian operator as well as  key  integral inequalities and some notations.
\subsection{Function spaces}
In this subsection,   we first define the function space
$$ H^s_N(\Omega) = \left\{ u = \sum_{k=0}^{\infty} c_k \varphi_k \ : \ \sum_{k=0}^{\infty} c_k^2(1 + \lambda_k^s) < +\infty \right\},  $$
where $\{\lambda_k, \varphi_k\}_{k \geq 0}$ are the eigenvalues and corresponding eigenfunctions of the Laplacian operator in $\Omega$ with Neumann boundary conditions on $\partial \Omega$:
\begin{equation*}
\begin{cases}
-\Delta \varphi_k = \lambda_k \varphi_k, & \text{in } \Omega, \\
\nabla \varphi_k \cdot \vec{n} = 0, & \text{on } \partial \Omega, \\
\|\varphi_k\|_2 = 1,
\end{cases}
\end{equation*}
with the norm
\begin{equation*}
\| u \|^2_{H^s_N(\Omega)} = \sum_{k=0}^{\infty} c_k^2(1 + \lambda_k^s) = \|u\|^2_{2} + [u]^2_{s,\Omega,N},
\end{equation*}
which is equivalent to the following norm
\begin{equation*}
\| u \|^2_{H^s_N(\Omega)} = \left( \int_{\Omega} u \, dx \right)^2 + [u]^2_{s,\Omega,N},
\end{equation*}
where the homogeneous norm is given by
\begin{equation*}
[u]_{s,\Omega,N}^2 = \sum_{k=0}^{\infty} c_k^2 \lambda_k^s.
\end{equation*}

Noticing that $ [u]^2_{0,\Omega,N} = \|u\|^2_2$.
From \cite{MSBA}, the space $ H^s_N(\Omega) $ coincides with the classical fractional Sobolev space $ H^s(\Omega) $ for $ s \in (0, 3/2) $. In the case $ s \in (3/2, 7/2) $, we have $ H^s_N(\Omega) := \{ u \in H^s(\Omega) : \nabla u \cdot \vec{n} = 0 \text{ on } \partial \Omega \} $. In any case, we have equality of the norms
$\|u\|_{H^s_N(\Omega)} = \|u\|_{H^s(\Omega)}~\text{for any } u \in H^s_N(\Omega). $ For more properties about these spaces,  the interested reader refers to \cite{EDNG, CIAM, RT}.

\subsection{Fractional Laplacian operator}
This subsection is devoted to stating the definition of the nonlocal operator and its corresponding properties.  Now, we define
the operator $I=-\left( -\Delta \right)^s $ with $s \in (0, 1)$
as follows:
$$I: \sum_{k=0}^{+\infty} c_k \varphi_k \longrightarrow -\sum_{k=0}^{+\infty} c_k \lambda_k^{s} \varphi_k  \text{~~which~maps~~} H^{2s}_N(\Omega) \text{~~onto~~} L^2(\Omega).$$
Without any proof, we give several key properties regarding the nonlocal operator $I.$
First, the operator $I$ can  be represented as a singular integral operator.
\begin{proposition}$^{\text{\cite{RT}}}$\label{qiyi}
	Consider a smooth function $ u\colon \Omega \to \mathbb{R} $. Then for all $ x \in \Omega $,
	\begin{equation}\label{addnewfor}
	I(u)(x) = \int_{\Omega} \bigl(u(y) - u(x)\bigr) K(x, y)\, dy,
	\end{equation}
	where $ K(x, y) $ is defined as follows. For all $ x, y \in \Omega, $
	\begin{equation*}
	K(x, y) = c_{s} \sum_{k \in \mathbb{Z}} \left( \frac{1}{|x - y - 2k|^{1+2s}} + \frac{1}{|x + y - 2k|^{1+2s}} \right),
	\end{equation*}
	where $ c_{s} $ is a constant depending only on $ s $.
\end{proposition}

The next propositions establish the commutativity property of the operator $I$.
\begin{proposition}$^{\text{\cite{RT}}}$\label{forall}
	~\\
	\noindent{\rm 1.} For all $u \in H^s_N(\Omega)$, we have
	$[u]^2_{s,\Omega} = -\langle I(u), u \rangle_{L^2(\Omega)};$  \\
	\noindent{\rm 2.} For all $u \in H^{2s}_N(\Omega)$, we have
	$[u]^2_{2s,\Omega} = \int_\Omega (I(u))^2 \, dx;$\\
	\noindent{\rm 3.} For any $u \in H^{s+1}_N(\Omega)$, we have   $[u]^2_{s+1,\Omega} = -\int_{\Omega} \partial_x I(u) \partial_x u \, dx;$\\
	\noindent{\rm 4.} For any $u \in H^{2s+1}_N(\Omega)$, we have
	$[u]^2_{2s+1,\Omega} = \int_{\Omega} \left(\partial_x I(u)\right)^2 \, dx.$
\end{proposition}
\begin{proposition}$^{\text{\cite{NDSL}}}$\label{part}
	Let $ s_1, s_2 \in [0, +\infty) $. If $ u \in H^{s_1 + s_2}_N(\Omega) $ and $ v \in H^{s_2}_N(\Omega) $, then
	\begin{equation*}
	\int_{\Omega} (-\Delta)^{s_1} u (-\Delta)^{s_2} v \, dx = \int_{\Omega} ((-\Delta)^{s_1 + s_2} u) v \, dx.
	\end{equation*}
\end{proposition}


Below, we give some important inequalities and propositions that play an important role in our proofs.
Note that we have from \cite{RT}
\begin{equation*}
\int_{\Omega} u \, dx \leq C(\Omega) \| u \|_2  \quad \text{(H${\rm \ddot{o}}$lder's inequality)},
\end{equation*}
\begin{equation*}
\| u \|_2 ^2 \leq C(\Omega) \| (-\Delta)^{\frac{s}{2}} u \|^2_2  \leq  C(\Omega) [u]^2_{s,\Omega}  \quad \text{(Fractional Poincar\'e's inequality)}.
\end{equation*}
In addition to aforementioned inequalities, the subsequent propositions  further give some embedding inequalities regarding Fractional Sobolev spaces.
\begin{proposition}$^{\text{\cite{NDSL}}}$\label{embed}
	Let $ s \in [0, +\infty) $. The following relationships hold:
	
	\noindent{\rm 1.} If $ s < d/2 $, then there exists a constant $ C $ such that
	\begin{equation*}
	\|u\|_{2d/(d-2s)} \leq C \|u\|_{H^s(\Omega)}, \quad \text{for all } u \in H^s(\Omega).
	\end{equation*}
	\noindent{\rm 2.} If $ s = d/2 $ and $ p \in (1, +\infty) $, then there exists a constant $ C_p$ such that
	\begin{equation*}
	\|u\|_{p} \leq C_p \|u\|_{H^s(\Omega)}, \quad \text{for all } u \in H^s(\Omega).
	\end{equation*}
	\noindent{\rm 3.} If $ s > d/2 $ and $ s - d/2 \notin \mathbb{N} $, then there exists a constant $ C $ such that
	\begin{equation*}
	\|u\|_{C^{s - d/2}(\Omega)} \leq C \|u\|_{H^s(\Omega)}, \quad \text{for all } u \in H^s(\Omega) .
	\end{equation*}	
	In particular, it holds
	\begin{equation*}
	\|u\|_{\infty} \leq C \|u\|_{H^s(\Omega)}, \quad \text{for all } u \in H^s(\Omega) .
	\end{equation*}
\end{proposition}

\begin{proposition}$^{\text{\cite{NDSL}}}$\label{interpolation}
	If $ s_0, s, s_1 \in [0, +\infty) $, $ s_0 \le s \le s_1 $ and $ u \in H_N^{s_1}(\Omega) $, then
	\begin{equation*}
	[u]_{s,\Omega,N} \le [u]^{1-\theta}_{s_0,\Omega,N} [u]^{\theta}_{s_1,\Omega,N},
	\end{equation*}
	where $ \theta = \frac{s - s_0}{s_1 - s_0}. $
\end{proposition}

Finally, we present the following two lemmas  to prove the existence of the solution.
\begin{lemma}$^{\text{\cite{RT}}}$\label{RT1}
	For all $ g \in L^2(\Omega) $ with $ \int_{\Omega} g \,dx = 0 $, there exists a unique function $ u \in H^{2s}_N(\Omega) $ such that	 
	\begin{equation*}
	-I(u) = g \quad \text{in } L^2(\Omega) \quad \text{and} \quad \int_{\Omega} u \, dx = 0.
	\end{equation*}	
	Furthermore, if $ g \in H^1(\Omega) $, then $ u \in H^{2s + 1}_N(\Omega) $.
\end{lemma}

\begin{lemma}$^{\text{\cite{RT}}}$\label{RT2}
	For all $ g \in L^2(\Omega) $, there exists a unique function $ v \in H^{2s}_N(\Omega) $ such that
	
	\begin{equation*}
	-I(v) + \frac{1}{|\Omega|}\int_{\Omega} v \, dx = g.
	\end{equation*}	
	In addition, if $ g \in H^1(\Omega) $, then $ v \in H^{2s + 1}_N(\Omega) $ and the map $ g \to v $ is bijective.
\end{lemma}

\subsection{Energy/entropy methods -- Integral inequalities}
For the solutions to problem \eqref{222},  in general, the maximum principle and comparison principle do not hold. In \cite{CIAM, RT}, the authors established the existence of non-negative solutions by relying on two types of integral inequalities. It is straightforward to show that the smooth solutions to \eqref{222} satisfy the energy inequality
\begin{equation}
-\int_{\Omega} u(t) I(u(t)) \, dx + \int_0^T \int_{\Omega} u^n (\partial_{x}I(u))^2 \, dx dt \leq -\int_{\Omega} u_0 I(u_0) \, dx,
\end{equation}
and entropy inequality
\begin{equation}
\int_{\Omega} G(u(t)) \, dx - \int_0^T \int_{\Omega} \partial_{x}u \partial_{x} I(u) \, dx dt = \int_{\Omega} G(u_0) \, dx,
\end{equation}
where  entropy function $ G : \mathbb{R} \rightarrow [0, +\infty) $ defined by the properties $G''(z) = \frac{1}{z^n} $ for any $z > 0 $, i.e.
\begin{align*}
G(z) = \int_A^z \int_A^r \frac{1}{t^n} \, dt \, dr,
\end{align*}
and  $A>0$, so that $ G $ is a non-negative convex function with a minimum point at $ z = A $  satisfying $ G(A) = G'(A) = 0 $.  A simple computation shows that for any $z>0$,
\begin{align*}
G(z) := \begin{cases}
\frac{z^{2-n}}{(n-1)(n-2)} + \frac{A^{1-n}z}{n-1} + \frac{A^{2-n}}{2-n}, & \text{if } n \neq \{1, 2\}, \\
z \ln\left(\frac{z}{A}\right) - z + A, & \text{if } n = 1, \\
\frac{z}{A} - \ln\left(\frac{z}{A}\right) - 1, & \text{if } n = 2,
\end{cases}
\end{align*}
and  $G(z) = +\infty $ for any $ z < 0 $, $ G(0) = \lim\limits_{z \to 0^+} G(z) $.

\subsection{Notations}
In what follows, we denote by $\| \cdot \|_r (r \ge 1) $ the norm in $L^r(\Omega)$ and by $\left\langle \cdot , \cdot\right\rangle  $ the $L^2(\Omega)$-inner product.
$C$ denotes a generic positive constant, which may differ at each appearance.

$\mathcal{D}(\Omega)$ denotes the space of test functions on an open domain $\Omega \subseteq \mathbb{R}^n$, consisting of all infinitely differentiable functions $\phi : \Omega \to \mathbb{R}$ with compact support. Formally,
\[
\mathcal{D}(\Omega) := \left\{ \phi \in C^\infty(\Omega) : \operatorname{supp}(\phi) \text{ is compact and } \operatorname{supp}(\phi) \subset \Omega \right\},
\]
where $\operatorname{supp}(\phi) = \overline{\{ x \in \Omega : \phi(x) \neq 0 \}}$.

Finally, we define $ z_+ = \max\{ 0, z \} $, $\Omega_T=(0,T)\times \Omega$, $\bar{u}_0:= \frac{1}{|\Omega|} \int_{\Omega} u_0(x) \,dx  $, and $\frac{1}{0}=+\infty$.


\section{Some key lemmas}


In this section, we present several lemmas that are essential for determining the long-time behavior of solutions. We begin with the following weighted interpolation inequality, whose proof is inspired by \cite[Lemma 1]{TAL}. 
\begin{lemma}\label{2007lem}
	For any measurable $ w : \Omega \to [0, \infty) $ and any non-negative $ v \in H^{2s+1}(\Omega) $, then for any $x_0 \in \Omega$, one has
	\begin{align}
	\int_\Omega \frac{v^2(x)}{w(x)}  \,dx  	&\int_\Omega w(x) |\partial_x (-\Delta)^s v|^2 \, dx
	\geq \frac{1}{4|\Omega|^2} \left( \int_\Omega | (-\Delta)^{\frac{s}{2}} v|^2\, dx   \right) ^2.
	\end{align}	
\end{lemma}
\begin{proof}
	By the Schwarz inequality, we obtain
	\begin{align}\label{MM}
	\int_{x_0}^x v\partial_y (-\Delta)^s v \,dy
	\le \left( \int_\Omega \frac{v^2}{w}  \,dx \right)^{\frac{1}{2}} \left( \int_\Omega w | \partial_x (-\Delta)^{s} v|^2  \,dx \right)^{\frac{1}{2}} :=M,		
	\end{align}
	for all $ x_0, x \in \Omega $. Then Applying Proposition \ref{forall} yields
	\begin{align*}
	M&\ge  v (-\Delta)^s v \Big| ^x_{x_0}
	-  \int_{x_0}^x (-\Delta)^s v\partial_y v \,dy   \nonumber\\
	&\ge v (-\Delta)^s v \Big| ^x_{x_0}
	- \int_{x_0}^x \frac{1}{2} \partial_{y}\left(| (-\Delta)^{\frac{s}{2}} v |^2 \right) 	\,dy \nonumber\\
	& = v(x) (-\Delta)^s v(x) - v(x_0) (-\Delta)^s v(x_0) - \frac{1}{2} \left| (-\Delta)^{\frac{s}{2}} v (x)  \right| ^2
	+\frac{1}{2} \left| (-\Delta)^{\frac{s}{2}} v (x_0)  \right|^2 .
	\end{align*}
	Integrating in $ x $ over $ \Omega $ gives
	\begin{align}\label{jj}
	|\Omega|M
	\ge \frac{1}{2} \int_\Omega | (-\Delta)^{\frac{s}{2}} v |^2 \,dx
	-|\Omega| v(x_0) (-\Delta)^s v(x_0)
	+ \frac{|\Omega|}{2} \left| (-\Delta)^{\frac{s}{2}} v (x_0)  \right|^2 .
	\end{align}
	
	Next, we claim  that there  exists $x_0\in \bar{\Omega}$ satisfying $ (-\Delta)^s v(x_0)\le 0.$
	Indeed,  if $v(x)$ attains the minimal value at either the endpoint $a$ or $b$, then \eqref{addnewfor} indicates
	\begin{align*}
	(-\Delta)^s v(a)\le 0~~\text{or}~~(-\Delta)^s v(b)\le 0.
	\end{align*}
	In this case, we can choose the endpoint $a$ or $b$ as $x_0$ to complete this proof.
	If $a$ and $b$ are not minimal value points, then the continuity of $v(x)$ shows that there exists $x_1 \in \Omega$ such that $v(x_1) = \min_{x \in \bar{\Omega}} v(x)$.
	Using \eqref{addnewfor} again yields
	\begin{align*}
	(-\Delta)^s v(x_1)\le 0.
	\end{align*}
	We therefore set $x_0=x_1,$  reducing \eqref{jj} to
	\begin{align*}
	|\Omega|M
	\ge \frac{1}{2} \int_\Omega | (-\Delta)^{\frac{s}{2}} v |^2 \,dx.
	\end{align*}
	Then the proof of lemma \ref{2007lem} is complete. 	
\end{proof}

Next, based on \cite{NJ,RTE}, we recall the following estimates of solutions to  the non-autonomous differential inequality. 
\begin{lemma}\label{inequ1}
	If $\alpha>0$, $\beta>0$,  then the non-negative solution $X(t)$ to 
	\begin{equation}\label{05.1}
	\begin{cases}
	X'(t)  +\beta X^\lambda(t) \le \alpha, \quad t> 0 ,\\
	X(0)=X_0>0,
	\end{cases}
	\end{equation}
	satisfies the following inequality:
	
	{\rm(1)} If $\lambda>1$, then
	\begin{equation*}
	X(t) \leq \begin{cases}
	\left(\frac{\alpha}{\beta} \right)^{\frac{1}{\lambda}}, & \text{if }X_0\le \left(\frac{\alpha}{\beta} \right)^{\frac{1}{\lambda}}, \\	
	\left(\frac{\alpha}{\beta} \right)^{\frac{1}{\lambda}}+ \left( \left( X_0-\left(\frac{\alpha}{\beta} \right) ^{\frac{1}{\lambda}}\right)^{1-\lambda}+ \beta(\lambda-1)t
	\right) ^{\frac{1}{1-\lambda}} ,
	& \text{if }X_0 > \left(\frac{\alpha}{\beta} \right)^{\frac{1}{\lambda}}.
	\end{cases}
	\end{equation*}
	
	{\rm(2)} If $0<\lambda<1$, then
	\begin{equation*}
	X(t) \leq \begin{cases}
	\left(\frac{\alpha}{\beta} \right)^{\frac{1}{\lambda}}, & \text{if }X_0\le \left(\frac{\alpha}{\beta} \right)^{\frac{1}{\lambda}}, \\	
	\frac{\alpha}{\beta}X_0^{1-\lambda}+ \left( X_0-\frac{\alpha}{\beta}X_0^{1-\lambda}\right) e^{-\frac{\beta t}{X_0^{1-\lambda}}},
	& \text{if }X_0 > \left(\frac{\alpha}{\beta} \right)^{\frac{1}{\lambda}}.
	\end{cases}
	\end{equation*}
	
\end{lemma}

Furthermore, for the differential inequality \eqref{05} with a time-dependent inhomogeneous term, we establish bilateral estimates of the solution.  To begin, the following comparison principle is introduced.
\begin{lemma}\label{Lem01}  Let $X(t)$ be the  non-negative solution of
	\begin{equation}\label{02}
	\begin{cases}
	X'(t)+\beta X^{\lambda}(t)= \alpha(t)\geq0,\quad t>t_{0}\geq0,\\
	X(t_{0})=X_{t_0}>0.
	\end{cases}
	\end{equation}
	And
	$X_{1}(t)$ be the  solution of
	\begin{equation}\label{01}
	\begin{cases}
	X_{1}'(t)+\beta X_{1}^{\lambda}(t)=0,\quad t>t_{0}\geq0,\\
	X(t_{0})=X_{t_0}>0.
	\end{cases}
	\end{equation}	
	Then for any $t\geq t_{0}$,
	\begin{equation}\label{com}
	X(t)\geq X_{1}(t)=\begin{cases}
	X_{t_0}e^{-\beta (t-t_0)},  &\text{~if~}\lambda=1,\\
	X_{t_0}\big[1+(\lambda-1)\beta X_{t_0}^{\lambda-1}(t-t_{0})\big]^{\frac{1}{1-\lambda}},  ~&\text{~if~} \lambda \neq 1.
	\end{cases}
	\end{equation}
\end{lemma}

\begin{proof}
	First,   since the equation \eqref{01}$_1$ is separable, it is not hard to obtain the right side of  \eqref{com} by solving equation $X_1^{-\lambda}X'_1=-\beta.$
	Next, set $h(t)=X(t)-X_{1}(t),$ then $h(t_{0})=0.$ We claim that
	\begin{center}
		$h(t)\geq0,$  $~~~\forall t\geq t_{0}.$
	\end{center}
	If not, there exists $t_{1}>t_{0}$ satisfying $h(t_{1})<0$. Define $t_{2}=\sup\{t\in[0,t_{1}),h(t)=0\}.$ Hence
	\begin{center}
		$h(t_{2})=0$ and $h(t)<0,$  $~~~\forall t_{2}<t<t_{1}$.
	\end{center}
	Therefore, it is easy to prove that
	\begin{center}
		$h^{'}(t)=\alpha(t)+\beta X_{1}^{\lambda}-\beta  X^{\lambda}\geq0$,  $~~~\forall t_{2}<t<t_{1},$
	\end{center}
	which implies
	$$h(t)\geq h(t_{2})=0.$$
	This is a contradiction. The proof is complete.
\end{proof}


Based on Lemma \ref{Lem01} stated above, we derive 
the following key lemma.

\begin{lemma}\label{inequ}
	If $\alpha(t)>0$, $\beta>0$,  then the non-negative solution $X(t)$ to 
	\begin{equation}\label{05}
	\begin{cases}
	X'(t)  +\beta X^\lambda(t) \le \alpha(t), \quad t> 0 ,\\
	X(0)=X_0>0,
	\end{cases}
	\end{equation}
	satisfies 
	satisfies
	\begin{align*}
	X(t)\leq X_0+\int_{0}^{\infty} \alpha(t)\, dt := M_{0}<+\infty.
	\end{align*}
	Further, the solution $X(t)$ 
	satisfies the following inequality:
	
	{\rm(1)} If $\lambda>1$, then
	\begin{align}\label{44}
	X_0&\left( 1+(\lambda-1)\beta X_0^{\lambda-1}t\right) ^{\frac{1}{1-\lambda}} \nonumber\\
	&\qquad\le X(t) \leq
	M_{0}\left(  1+\beta (\lambda-1)(1-\kappa)  M_{0}^{\lambda-1}t\right)  ^{\frac{1}{1-\lambda}}
	+\int_{\kappa t}^{t}\alpha(r)\,dr.
	\end{align}

	(2) If $ 0<\lambda\le 1$, then
	\begin{align}\label{44.1}
	X_0e^{-\beta M_{0}^{\lambda-1} t}
	\le X(t)
	\leq M_{0}e^{-\beta(1-\kappa) M_{0}^{\lambda-1}t}
	+ \int_{\kappa t}^{t}\alpha(r)e^{-\beta(1-\kappa) M_{0}^{\lambda-1}(t-r)}\,dr,
	\end{align}	
	where $\kappa\in (0,1).$

\end{lemma}

\begin{proof}
	First, integrating the equation $(\ref{05})_1$ from $0$ to $t$, we get
	\begin{align*}
	X(t)\le X_{0}+\int_0^{t}\alpha(r)\,dr-\beta\int_0^{t}X^{\lambda}(r)\,dr
	\leq X_{0}+\int_{0}^{\infty}\alpha(r)\,dr:=M_{0}<+\infty.
	\end{align*}

	Next, we prove \eqref{44} and \eqref{44.1}, there are two cases.
	{\bf Case 1.} If $\lambda>1$, for any $t>0$, let $X_{2}(t)$ be the solution of
	\begin{equation}\label{04}
	\begin{cases}
	X_{2}^{'}(r)+\beta X_{2}^{\lambda}(r)=0,\quad r\geq\kappa t,\\
	X_{2}\left( \kappa t\right) =X(\kappa t),
	\end{cases}
	\end{equation}
	where $\kappa\in (0,1).$ 
	From $X^{'}(r)\le-\beta X^{\lambda}(r)+\alpha(r)$ and Lemma \ref{Lem01}, then we have
	\begin{align*}
	X(t)-X(\kappa t)&=\int_{\kappa t}^{t}X^{'}(r)\,dr
	\le-\beta \int_{\kappa t}^{t}X^{\lambda}(r)\,dr+\int_{\kappa t}^{t}\alpha(r)\,dr\\
	&\leq\int_{\kappa t}^{t}\alpha(r)\,dr-\beta \int_{\kappa t}^{t}X_{2}^{\lambda}(r)\,dr
	=\int_{\kappa t}^{t}\alpha(r)\,dr+\int_{\kappa t}^{t}X_{2}^{'}(r)\,dr\\
	&=X_{2}(t)-X_{2}(\kappa t)+\int_{\kappa t}^{t}\alpha(r)\,dr,
	\end{align*}
	which implies
	\begin{align}\label{y1.1}
	X(t)\leq X_{2}(t)+\int_{\kappa t}^{t}\alpha(r)\,dr.
	\end{align}
	Noting that
	\begin{align}\label{y1.2}
	X_{2}(t)&=X(\kappa t)\left[ 1+\beta(\lambda-1)(1-\kappa) X^{\lambda-1}(\kappa t)t\right] ^{\frac{1}{1-\lambda}}\leq M_{0}\left[ 1+\beta(\lambda-1)(1-\kappa) M_{0}^{\lambda-1} t\right] ^{\frac{1}{1-\lambda}}.
	\end{align}
	Combining \eqref{y1.1}, \eqref{y1.2}, and Lemma \ref{Lem01}, we conclude
	$$	X_0\big[1+(\lambda-1)\beta X_0^{\lambda-1}t\big]^{\frac{1}{1-\lambda}} \le X(t)\leq
	M_{0}\left[ 1+\beta (\lambda-1)(1-\kappa)  M_{0}^{\lambda-1}t\right] ^{\frac{1}{1-\lambda}}
	+\int_{\kappa t}^{t}\alpha(r)\,dr.$$
	
	{\bf Case 2.} If $0<\lambda\le 1$, then for any $t>0$, noticing that $X^{\lambda}(t)=X^{\lambda-1}(t)X(t)\geq M_{0}^{\lambda-1}X(t)$, we get
	\begin{equation}\label{06}
	\begin{cases}
	X^{'}(t)+\beta M_{0}^{\lambda-1}X(t)\leq \alpha(t),\\
	X(\kappa t)=X(\kappa t)>0.
	\end{cases}
	\end{equation}
	Since $(\ref{06})_1$ is a linear differential inequality, it follows from Lemma \ref{Lem01} that there exists an explicit formula for the solution
	\begin{align*}
	X_0e^{-\beta M_{0}^{\lambda-1} t}
	\le X(t)
	\leq X(\kappa t)e^{-\beta(1-\kappa) M_{0}^{\lambda-1}t}
	+ \int_{\kappa t}^{t}\alpha(r)e^{-\beta(1-\kappa) M_{0}^{\lambda-1}(t-r)}\,dr.
	\end{align*}
	The proof of Lemma \ref{inequ} is complete.

\end{proof}

\section{Time-and space-dependent forces $S(t, x)$}

In this section,  we regularize the degenerate term of the thin-film equation and discretize it in time to establish the existence of solutions to the steady-state problem.
Subsequently, the existence of solutions to the regularized problem is obtained via energy/entropy estimates and weak convergence techniques.
Furthermore, we establish a relationship between the dissipation functional  and the energy functional  (Lemma \ref{2007lem}).
Building on this, a refined analysis of the obtained energy inequality enables us to obtain obtain novel bilateral estimates for the convergence rate of the solution.

\subsection{Regularized problem}
To prove Theorem \ref{ex1}, we consider the following regularized problem:
\begin{equation}\label{Regular1}
\begin{cases}
\partial_t u + \partial_x\left(f_\epsilon(u) \partial_x I(u) \right) = S(t, x),~~&(t,x) \in (0,T)\times \Omega,\\
\partial _x u =0,~f_\epsilon(u) \partial_x I(u) = 0,~~&(t,x) \in (0,T) \times \partial \Omega,\\
u(0, x)=u_{0}(x),~~&x \in \Omega,
\end{cases}
\end{equation}
where $f_\epsilon(z) = z_+^n+\epsilon$, $\epsilon > 0$  and $0 < s < 1$. 
For problem \eqref{Regular1}, we have the following result.
\begin{proposition}\label{exist.regular}
	Let $n \ge 1$, $s \in (0,1)$.  
	If $u_0 \in H^{s}(\Omega)$, $S(t, x) \in L^1(0, \infty;H^s(\Omega) ) \cap  L^1(0, \infty;L^1(\Omega) )$, then for all $T > 0$, there exists the solution $u_\epsilon$ to problem \eqref{Regular1} satisfying
	$$u_\epsilon \in L^\infty(0, T; H^{s}(\Omega)) \cap L^2(0, T; H^{2s + 1}_{N}(\Omega)) $$
	and
	\begin{align}\label{exist.defi}
	\iint_{\Omega_T} u_\epsilon \partial_t \varphi  \,dx dt + \iint_{\Omega_T} f_\epsilon(u_\epsilon) \partial_x I(u_\epsilon) \partial_x \varphi \,dx dt = -\int_{\Omega} u_0 \varphi  (0,\cdot) \,dx - \iint_{\Omega_T} S(t, x) \varphi  \,dx dt,
	\end{align}
	for all $\varphi \in C^1(0, T; H^1(\Omega))$ with support in $(0, T) \times \bar{\Omega}$.
	
	Further, for almost every $t \in (0, T)$, $u_\epsilon$ fulfills
	\begin{equation}\label{exist.mass}
	\int_{\Omega} u_\epsilon(t, x) \, dx = \int_{\Omega} u_0(x)\,dx + \int_{0}^t  \int_{\Omega} S(r, x) \,dxdr,
	\end{equation}
	and
	\begin{equation}\label{exist.energy}
	[ u_\epsilon ]^2_{s,\Omega} + 2\int_{0}^{T} \int_{\Omega} f_\epsilon(u_\epsilon) \left| \partial_{x} I(u_\epsilon)\right| ^{2}  \,dx dt \leq [ u_0 ]^2_{s,\Omega} -2 \int_{0}^{T} \int_{\Omega} S(t, x) I(u^\epsilon)  \, dx dt.
	\end{equation}
	
	Additionally,  the following entropy estimate holds
	\begin{align}\label{exist.entropy}
	\int_{\Omega} G_\epsilon(u_{\epsilon}) \, dx + \int_0^{t} [ u_\epsilon]_{s+1,\Omega,N}^{2} \, dr \leq \int_{\Omega} G_\epsilon(u_0) \, dx
	+\int_{0}^{t} \int_{\Omega} S(r, x) G'_\epsilon(u_\epsilon)  \, dx dr,
	\end{align}
	where $G_\epsilon$ is a non-negative function such that $G''_\epsilon(z)=\frac{1}{f_\epsilon(z)}.$
\end{proposition}

The proof of Proposition \ref{exist.regular} is divided into two steps. In the next subsection, we consider the steady-state problem associated with \eqref{Regular1} and use pseudo-monotone operator theory to prove the existence of its solutions. Then, in the third subsection, we derive necessary a priori estimates via entropy/energy estimates  and apply weak convergence techniques to establish the global existence of solutions to the regularized problem \eqref{exist.regular}.

\subsection{Steady-state problem}
In this subsection, we consider the steady-state problem
\begin{equation}\label{Stationary}
\begin{cases}
u + \tau \partial_{x} \left( h(u) \partial_{x} I(u) \right)-\tau S(x) =g,  ~~& \text{in } \Omega,\\
\partial_{x} u = 0, \quad \partial_{x} I(u) = 0,  ~~&\text{on } \partial \Omega,
\end{cases}
\end{equation}
where $\tau > 0,$ $g \in H^{s}(\Omega)$,   $h(s):(-\infty,+\infty)\rightarrow(0,+\infty)$ is a continuous function which is bounded below by a constant $c>0.$
The following proposition will provide us some results about problem \eqref{Stationary}.
\begin{proposition}\label{Stationary.pro}
	If $g \in H^{s}(\Omega)$, $S(x) \in H^{s}(\Omega)$, then there exists $u \in H^{2s + 1}_{N}(\Omega)$ of problem \eqref{Stationary} which satisfies for all $\varphi \in H^{1}(\Omega)$,
	\begin{align}\label{Stationary.defi}
	\int_{\Omega} u \varphi  \,dx - \tau \int_{\Omega} h(u) \partial_{x} I(u) \partial_{x} \varphi  \,dx -\tau \int_{\Omega} S(x) \varphi \,dx= \int_{\Omega} g \varphi  \,dx.
	\end{align}
	Meanwhile, $u$ satisfies
	\begin{align}\label{Stationary.mess}
	\int_{\Omega} u(x)  \,dx = \int_{\Omega} g(x)  \,dx+\tau \int_{\Omega} S(x)\, dx,
	\end{align}
	and
	\begin{align}\label{Stationary.energy}
	[u]^2_{s,\Omega} + 2\tau \int_{\Omega} h(u) \left| \partial_{x} I(u)\right| ^{2} \, dx \leq [ g ]^2_{s,\Omega} -2\tau \int_{\Omega} S(x)I(u) \,dx.
	\end{align}
	Further,  $u$ also satisfies
	\begin{equation}\label{Stat.entropy}
	\int_{\Omega} H(u) \,dx + \tau [u]^2_{s+1,\Omega,N}
	\le \int_{\Omega} H(g) \,dx +\tau \int_{\Omega}S(x)H'(u) \, dx,
	\end{equation}
	where $H(u)$ is a non-negative function with $H'' (z) = \frac{1}{h(z)}$.	
\end{proposition}

\begin{proof}
	Following the approach in \cite[Proposition  6]{CIAM}, we apply the Lemma \ref{RT2} to recover all test functions from $H^1(\Omega) $ by considering
	$$
	\varphi = -I(v) + \frac{1}{|\Omega|} \int_{\Omega} v \, dx
	$$
	for some function $ v \in H^{2s + 1}_N(\Omega) $. Thus,  \eqref{Stationary.defi} becomes
	\begin{align}\label{Stationary.def1}
	-\int_{\Omega} u I(v) \, dx &+ \frac{1}{|\Omega|}\left( \int_{\Omega} u \, dx \right) \left( \int_{\Omega} v \, dx \right) + \tau \int_{\Omega} h(u) \partial_x I(u) \partial_x I(v) \, dx  \nonumber\\
	&= -\int_{\Omega} g I(v) \, dx + \frac{1}{|\Omega|} \left( \int_{\Omega} g \, dx \right) \left( \int_{\Omega} v \, dx \right) \nonumber\\
	&\quad-\tau \int_{\Omega} SI(v) \,dx + \frac{\tau}{|\Omega|} \left( \int_{\Omega} S \, dx \right) \left( \int_{\Omega} v \, dx \right).
	\end{align}
	Then, we introduce the following nonlinear operator
	\begin{equation*}
	\Gamma(u)(v) = -\int_{\Omega} u I(v) \, dx + \frac{1}{|\Omega|}\left( \int_{\Omega} u \, dx \right) \left( \int_{\Omega} v \, dx \right) + \tau \int_{\Omega} h(u) \partial_x I(u) \partial_x I(v) \, dx,
	\end{equation*}
	for all $ u, v \in H^{2s + 1}_N(\Omega)$. 
	It is not difficult to verify that $\Gamma(u)(v)$ is a continuous, coercive and pseudo-monotone operator; details are given in Appendix A.  
	Further, we define the functional $ T_g $ by the following form
	\begin{align*}
	T_g(v) = &-\int_{\Omega} gI(v) \, dx + \frac{1}{|\Omega|}\left( \int_{\Omega} g \, dx \right) \left( \int_{\Omega} v \, dx \right)\\
	&-\tau \int_{\Omega} SI(v) \,dx + \frac{\tau}{|\Omega|} \left( \int_{\Omega} S \, dx \right) \left( \int_{\Omega} v \, dx \right),
	\end{align*}
	for $ v \in H^{2s + 1}_N(\Omega) $, is a linear form on $ H^{2s + 1}_N(\Omega) $. Consequently, according to \eqref{Stationary.def1},
	problem \eqref{Stationary} simplifies to the following:
	\begin{equation}\label{Stationary.pro1}
	\begin{cases}
	\text{Let } V = H^{2s + 1}_N(\Omega). \\
	\Gamma : V \rightarrow V^* \text{ coercive, continuous and pseudo-monotone.} \\
	T_g \in V^*. \\
	\text{Find } u \in H^{2s + 1}_N(\Omega) \text{ such that } \Gamma(u) = T_g \text{ in } V^*.
	\end{cases}
	\end{equation}
	The theory of pseudo-monotone operators \cite{JLL} implies for \eqref{Stationary.pro1} so for all $g \in H^s(\Omega)$, there exists $ u \in H^{2s + 1}_N(\Omega) $ such that
	$$\Gamma(u)(v) = T_g(v) \quad \text{in } V^*.$$
	Then, by utilizing Lemma \ref{RT2}, we derive that there exists $u\in H^{2s + 1}_N(\Omega) $ that satisfies \eqref{Stationary.defi} for all $\varphi \in H^1(\Omega)$.
	
	Below, we explore the properties of $u$. First by taking $\varphi = 1$ in \eqref{Stationary.defi} we obtain mass equation \eqref{Stationary.mess}. Next, take $v = u - \frac{1}{|\Omega|}\int_{\Omega} u \, dx$ in \eqref{Stationary.def1}, we get
	\begin{align*}
	&[u]^2_{s,\Omega} + \tau \int_{\Omega} h(u)|\partial_x I(u)|^2  \, dx
	= - \int_{\Omega} g I(u) \, dx
	-\tau  \int_{\Omega} S I(u) \, dx  \nonumber\\
	\leq &\,[g]^2_{s,\Omega} [u]_{s,\Omega} ^2 -\tau  \int_{\Omega} S I(u) \, dx
	\leq \frac{1}{2} [g]^2_{s,\Omega} + \frac{1}{2} [u]^2_{s,\Omega}  -\tau  \int_{\Omega} S I(u) \, dx ,
	\end{align*}
	which corresponds to  \eqref{Stationary.energy}. Lastly, since $H'$ is smooth with $H'$ and $H''$ are bounded, and $\Omega$ is bounded so we can take $\varphi = H'(u) \in H^1(\Omega)$ in \eqref{Stationary.defi} to obtain
	\begin{equation*}
	\int_{\Omega} u H'(u) \, dx - \tau \int_{\Omega} h(u) \partial_x I(u) \partial_x u H''(u) \, dx = \int_{\Omega} g H'(u) \, dx
	+ \tau  \int_{\Omega} S(x) H'(u) \, dx.
	\end{equation*}
	According to Proposition \ref{forall}, we further get
	\begin{equation*}
	\tau [u]^2_{s+1,\Omega,N}
	=  \int_{\Omega} H'(u)(g+\tau S(x)-u) \, dx
	\leq \int_{\Omega} (H(g) - H(u)) \, dx +\tau \int_{\Omega}S(x)H'(u) \, dx,
	\end{equation*}
	where $H$ is convex, so we deduce \eqref{Stat.entropy}.
	The proof of the Proposition \ref{Stationary.pro} is complete.
\end{proof}

\subsection{Proof of Proposition \ref{exist.regular}}
In this subsection, we establish the existence of  solutions to the regularized problem \eqref{Regular1}.
We begin to discrete problem \eqref{Regular1} with respect to $t$ and construct a  step function
$$u^\tau_\epsilon(t, x) = u^k(x), \quad \text{for } t \in [k\tau, (k+1)\tau), ~k \in \{0, \dots, N-1\}, ~\tau = \frac{T}{N}, $$
where $u^{k+1}(x)$ solves
\begin{align}\label{uk1}
u^{k+1} + \tau \partial_{x} \left( f_{\epsilon}(u^{k+1}) \partial_{x} I(u^{k+1}) \right) -\tau S^{k+1}(x)= u^k,
\end{align}
where $f_{\epsilon}(z) = z^n_++\epsilon, \epsilon > 0,$ and
$$S^\tau (t,x)= S^{k}(x) := S (k\tau,x), \quad \text{for } t \in [k\tau, (k+1)\tau), ~k \in \{0, \dots, N-1\}.  $$
The existence of the $u^k$ follows from Proposition \ref{Stationary.pro} with $h(z)=f_{\epsilon}(z)$ by induction on $k$ with $u^0 = u_0$.
We deduce the following result concerning the step function $u^\tau_\epsilon$.
\begin{corollary}\label{Stationary.ex}
	If $N > 0$ and $u_0 \in H^{s}(\Omega)$, $S^\tau(t, x) \in L^1(0, \infty;H^s(\Omega) ) \cap  L^1(0, \infty;L^1(\Omega) )$, then there exists a function $u^\tau_\epsilon \in L^\infty(0, T; H^{s}(\Omega))$ satisfying
	
	\noindent{\rm 1.} $t  \longrightarrow u^\tau_\epsilon(t,x) \text{ is constant on } [k\tau, (k+1)\tau), k \in \{0, \ldots, N-1\} \text{ and } \tau = \frac{T}{N}.$
	
	\noindent{\rm 2.} $u^\tau_\epsilon = u_0 \text{ on } [0, \tau) \times \Omega.$
	
	\noindent{\rm 3.} For all $t \in (0, T)$,
	\begin{align}\label{Stationary.mass1}
	\int_{\Omega} u^\tau_\epsilon(t, x)  \,dx = \int_{\Omega} u_0(x)  \,dx+ \tau \left[ \frac{t}{\tau} \right]  \int_{\Omega} S^\tau(t, x)\, dx  ,
	\end{align}
	where $\left[ \frac{t}{\tau} \right]$ denotes the greatest integer less than or equal to  $\frac{t}{\tau} $.	
	
	\noindent{\rm 4.} For all $\varphi \in C^1_c(0, T; H^1(\Omega))$,
	\begin{align}\label{Stationary.defi1}
	\iint_{\Omega _{\tau, T}}  \frac{u^\tau_\epsilon-S_\tau u^\tau_\epsilon}{\tau} \varphi  \,dx dt = \iint_{\Omega _{\tau, T}} f_\epsilon(u^\tau_\epsilon) \partial_{x} I(u^\tau_\epsilon) \partial_{x} \varphi  \,dx dt + \iint_{\Omega _{\tau, T}} S^\tau(t, x) \varphi \, dx dt,
	\end{align}
	where $S_{\tau} u^{\tau}_\epsilon(t, x) = u^{\tau}_\epsilon(t - \tau, x)$ and $\Omega_{\tau, T} = (\tau, T) \times \Omega.$
	
	\noindent{\rm 5.} For all $t \in (0, T)$,
	\begin{align}\label{Stationary.energy1}
	[u^\tau_\epsilon ]^2_{s,\Omega}  + 2\iint_{\Omega _T} f_\epsilon(u^\tau_\epsilon) \left| \partial_{x} I(u^\tau_\epsilon)\right| ^{2}  \,dx dt \leq [ u_0 ]^2_{s,\Omega}  - 2\iint_{\Omega _T} S^\tau(t, x) I(u^\tau_\epsilon)  \, dx dt.
	\end{align}
	
	\noindent{\rm 6.} For all $t \in (0, T)$,
	\begin{align}\label{Stationary.entropy1}
	\int_{\Omega} G_\epsilon(u^\tau_\epsilon) \, dx + \int_0^{t} [ u^\tau_\epsilon(r, \cdot) ]^2_{s+1,\Omega,N} \, dr \leq \int_{\Omega} G_\epsilon(u_0)  \,dx
	+\int_{0}^{t} \int_{\Omega} S^\tau(r, x) G'_\epsilon(u^\tau_\epsilon)   \,dx dr.
	\end{align}
\end{corollary}

{\bf Proof of Proposition \ref{exist.regular}.}
Some ideas of this proof come from \cite{CIAM,RT}. To make readers easily follow the outline of our proof, we will divide the proof  into  four steps.

\textbf{Step 1: A priori estimates.} We summarize the a priori estimates for $u^{\tau}_\epsilon$ with respect to $\tau$. First, we may claim there exists a constant $C>0$ independent of $\tau$ such that the solution $ u^{\tau}_\epsilon $ satisfying
\begin{align}
&\label{zxc1}\|u^{\tau}_\epsilon\|_{L^{\infty}(0,T;H^{s}(\Omega))} \leq C,\\
&\label{zxc2}\sqrt{ \epsilon } \|\partial_{x} I(u^{\tau}_\epsilon)\|_{L^{2}(0,T;L^{2}(\Omega))} \leq C, \\
&\label{zxc3}\left\|\frac{u^{\tau}_\epsilon - S_{\tau} u^{\tau}_\epsilon}{\tau}\right\|_{L^{2}(\tau,T;W^{-1,r}(\Omega))} \leq C,
\end{align}
where $ C $ does not depend on $ \tau > 0 $ and $r = \frac{2p}{p+2}$.

Indeed, estimates \eqref{Stationary.mass1} and \eqref{Stationary.energy1} indicate that $u^\tau_\epsilon$ is bounded in $L^\infty(0, T; H^{s}(\Omega))$ and $\partial_x I(u^\tau_\epsilon)$ is bounded in $L^2(0, T; L^2(\Omega))$.
Noting that $n \ge 1$ and the Lipschitz continuity of the function $f_\epsilon$, we know that $f_\epsilon(u^\tau_\epsilon)$ is bounded in $L^\infty(0, T; H^{s}(\Omega))$. By the Sobolev embedding theorem, we deduce that $f_\epsilon(u^\tau_\epsilon)$ is also bounded in $L^\infty(0, T; L^p(\Omega))$ with $p < \frac{2}{1 - 2s}$, if $0 < s \le  \frac{1}{2}$; $p =\infty$, if $\frac{1}{2} < s <1$, which indicates that $f_\epsilon(u^\tau_\epsilon) \partial_x I(u^\tau_\epsilon)$ is bounded in $L^2(0, T; L^r(\Omega))$ with  $\frac{1}{r} = \frac{1}{2} + \frac{1}{p}$.

Further, we fix $v \in W^{1,r'}(\Omega)$ with $\|v\|_{W^{1,r'}(\Omega)} \leq 1$.  By \eqref{uk1}, for any $t \in (0, T)$,  we have   	 $$\int_\Omega \frac{u^{\tau}_\epsilon - S_{\tau} u^{\tau}_\epsilon}{\tau} v \, dx = \int_\Omega f_\epsilon(u^\tau_\epsilon) \partial_{x} I(u^\tau_\epsilon) \partial_{x} v \, dx + \int_\Omega S^\tau(t, x) v \, dx.  $$
Since $\|v\|_{W^{1,r'}(\Omega)} \leq 1$, the last equality implies that  	
$$\int_\Omega \frac{u^{\tau}_\epsilon - S_{\tau} u^{\tau}_\epsilon}{\tau} v \, dx
\le \|f_\epsilon(u^\tau_\epsilon) \partial_{x} I(u^\tau_\epsilon)\|_{r}  + \|S^\tau(t, x)\|_{2} .  $$  	
Therefore, by the definition of dual norm,  we obtain	
$$  \left\|\frac{u^{\tau}_\epsilon - S_{\tau} u^{\tau}_\epsilon}{\tau}\right\|_{L^{2}(\tau,T;W^{-1,r}(\Omega))} \leq C,
$$
for a constant $C$ independent of $\tau$.

\textbf{Step 2: Compactness result.}
With the help of the following embedding relationships
\begin{equation}\label{embedding}
\begin{cases}
H^{s}(\Omega) \hookrightarrow L^p(\Omega) \hookrightarrow W^{-1,r}(\Omega),~~& \text{~if~} 0 < s \le  \frac{1}{2},\\
H^{s}(\Omega) \hookrightarrow C^{\frac{2s-1}{2}}(\Omega) \hookrightarrow W^{-1, 2}(\Omega),~~&\text{~if~} \frac{1}{2} < s <1,
\end{cases}
\end{equation}
and Aubin's lemma, it is not hard to verify that $u^\tau_\epsilon$ is relatively compact in $C(0, T; L^p(\Omega))$ with $p < \frac{2}{1 - 2s}$ if $0 < s \le  \frac{1}{2}$   and $C(0, T; C^{\frac{2s-1}{2}}(\Omega))$  if $\frac{1}{2} < s <1$.

Moreover, we observe that  $u^\tau_\epsilon$ is bounded in $L^2(0, T; H^{2s + 1}_N(\Omega))$ from \eqref{Stationary.energy1}. From the following embedding relationships
$$
H^{2s + 1}_N(\Omega) \hookrightarrow H^{s+1}_N(\Omega) \hookrightarrow W^{-1, r}(\Omega),
$$
we conclude that $u^\tau_\epsilon$ is relatively compact in $L^2(0, T; H^{s+1}_N(\Omega))$. Therefore, we can extract a subsequence, also denoted $u^\tau_\epsilon$, such that when $\tau$ approaches zero, we have
$u_\epsilon \in L^\infty(0, T; H^s(\Omega)) \cap L^2(0, T; H^{s+1}_N(\Omega))$ satisfying
\begin{equation}\label{partial}
\begin{cases}
u^\tau_\epsilon \to u_\epsilon \text{~in~} L^2(0, T; H^{s+1}_N(\Omega)) \text{~strongly},\\
\partial_x I(u^\tau_\epsilon) \rightharpoonup \partial_x I(u_\epsilon) \text{~in~} L^2(0, T; L^2(\Omega))  \text{~weakly}.
\end{cases}
\end{equation}
And then, on the one hand, if $0 < s \le  \frac{1}{2}$, then we have
$$u_\epsilon^\tau \to u_\epsilon  \text{~almost everywhere in~} \Omega_T.$$
On the other hand, if $\frac{1}{2} < s <1$, then Egorov's theorem implies that for any $\eta>0$, there exists a set $E_\eta \subset \Omega_T$ with $|\Omega_T \setminus E_\eta| \le \eta$ such that
\begin{equation}\label{local}
u_\epsilon^\tau \to u_\epsilon \text{~uniformly~in~} E_\eta.
\end{equation}

\textbf{Step 3: Derivation of definition \eqref{exist.defi}.}
Let us take the limit in \eqref{Stationary.defi1}. We first have
\begin{align*}
\iint_{\Omega_{ \tau,T}}  \frac{u^\tau_\epsilon - S_\tau u^\tau_\epsilon}{\tau} &\varphi\, dx  dt
= \int_0^T \int_{\Omega} u^\tau_\epsilon(t, x) \frac{ \varphi(t, x) - \varphi(t+\tau, x) }{\tau} \, dxdt  \nonumber\\
&  -\frac{1}{\tau} \int_0^\tau\int_{\Omega} u^\tau_\epsilon(t, x) \varphi(t, x) \, dxdt + \frac{1}{\tau} \int_{T-\tau}^T\int_{\Omega} u^\tau_\epsilon(t, x) \varphi(t + \tau, x) \, dxdt  \nonumber\\
& \longrightarrow -\iint_{\Omega_T} u_\epsilon \partial_t\varphi \, dxdt - \int_{\Omega} u_\epsilon(0, x) \varphi(0, x) \, dx, ~~\text{as~}  \tau \to 0.
\end{align*}
For the nonlinear term, noting that  $u^\tau_\epsilon$ is bounded in $L^2(0, T; H^{2s + 1}_N(\Omega))$, and then applying embedding theorem, we deduce the following convergence:
\begin{align}
&I(u^\tau_\epsilon) \rightarrow I(u_\epsilon) \text{ in } L^2(0, T; L^q(\Omega)) \text{ for all } q < \infty, \label{Iu}\\
&\partial_{x}u^\tau_\epsilon \rightarrow \partial_{x} u_\epsilon \text{ in } L^2(0, T; L^p(\Omega)) \text{ for all } p < \frac{2}{1 - 2s}. \nonumber
\end{align}
On the one hand, for $0<s\le \frac{1}{2}$, since $u^\tau_\epsilon \rightarrow u_\epsilon$ in $C^0(0, T; L^p(\Omega))$ with $p < \frac{2}{1 - 2s}$ and $f_\epsilon$ is Lipschitz, we have
$$
f_\epsilon(u^\tau_\epsilon) \rightarrow f_\epsilon(u_\epsilon) \text{ in } C^0(0, T; L^p(\Omega)) \text{ for all } p < \frac{2}{1 - 2s}.  $$
\noindent For the term $(u^\tau_\epsilon)^{n-1}_+$, we have
\begin{equation*}
\begin{cases}
(u^\tau_\epsilon)^{n-1}_+ \rightarrow (u_\epsilon)^{n-1}_+ \text{ in } C^0(0, T; L^p(\Omega)) \text{ for all } p < \frac{2}{1 - 2s},~~& \text{~if~}n \geq 2,\\
(u^\tau_\epsilon)^{n-1}_+ \rightarrow (u_\epsilon)^{n-1}_+ \text{ in } C^0(0, T; L^{\frac{p}{n-1}}(\Omega)) \text{ for all } p < \frac{2}{1 - 2s},~~&\text{~if~} 1<n < 2.
\end{cases}
\end{equation*}
Thus,  integrating by parts and the convergence above yield
\begin{align*}
&\iint_{\Omega_T} f_\epsilon(u^\tau_\epsilon) \partial_x I(u^\tau_\epsilon) \partial_{x} \varphi
=- \iint_{\Omega_T} f_\epsilon(u^\tau_\epsilon) I(u^\tau_\epsilon) \partial^2_{xx}\varphi
- \iint_{\Omega_T} n(u^\tau_\epsilon)_+^{n-1} \partial_x u^\tau_\epsilon I(u^\tau_\epsilon) \partial_x \varphi \nonumber\\
\rightarrow &-\iint_{\Omega_T} f_\epsilon(u_\epsilon)  I(u_\epsilon) \partial_{xx}^2 \varphi - \iint_{\Omega_T} n (u_\epsilon)_+^{n-1} \partial_x u_\epsilon I(u_\epsilon) \partial_{x} \varphi
= \iint_{\Omega_T} f_\epsilon(u_\epsilon) \partial_x I(u_\epsilon) \partial_{x} \varphi .
\end{align*}

On the other hand, for $\frac{1}{2}<s<1$, \eqref{local} shows
$$
f_\epsilon(u^\tau_\epsilon) \partial_x \varphi  \to f_\epsilon(u_\epsilon) \partial_x \varphi  \text{ in } L^2(0, T; L^2(\Omega)) \text{ strongly.}
$$
Thus, we collect \eqref{partial} to get
$$
\iint_{\Omega_T} f_\epsilon(u^\tau_\epsilon) \partial_x I(u^\tau_\epsilon) \partial_x \varphi \, dx dt \to \iint_{\Omega_T} f_\epsilon(u_\epsilon) \partial_x I(u_\epsilon) \partial_x \varphi \, dx  dt.
$$
Once again, by means of the definition of $S^\tau(t, x)$, we have
$$\iint_{\Omega_T} S^\tau(t, x) \varphi \, dx dt \to \iint_{\Omega_T} S(t, x) \varphi \, dx  dt.$$

\textbf{Step 4: The properties of $u_\epsilon$.}
For the properties of $u_\epsilon$, we first observe that the mass equation \eqref{exist.mass} follows from \eqref{Stationary.mass1} by $u^\tau_\epsilon \to u_\epsilon$ in $L^\infty(0, T; L^1(\Omega))$.

Next, we note that $u^\tau_\epsilon$ is bounded in $L^\infty(0, T; H^s(\Omega))$, then from Fatou's lemma, we obtain
$$[u_\epsilon]_{s,\Omega} \leq \liminf_{\tau \to 0} \,[ u^\tau_\epsilon ]_{s,\Omega}.  $$
It follows from estimate \eqref{Stationary.energy1} and \eqref{Iu} that $\sqrt{f_\epsilon(u^\tau_\epsilon)} \partial_{x} I(u^\tau_\epsilon)$  weakly converges in $L^2(0, T; L^2(\Omega))$ and $I(u^\tau_\epsilon)$ strongly converges in $L^2(0, T; L^q(\Omega))$ for all $q < \infty$, then the lower semi-continuity allows us to conclude \eqref{exist.energy}.

The fact that  $G_\epsilon(u^\tau_\epsilon) $ almost everywhere converges to $G_\epsilon(u_\epsilon)$  and Fatou's lemma guarantee for almost every  $t \in (0, T)$
$$\int_\Omega G_\epsilon(u_\epsilon(t, x)) \, dx \le \liminf_{\tau \to 0} \int_\Omega G_\epsilon(u^\tau_\epsilon(t, x)) \, dx.  $$
Besides, the fact that $u^\tau_\epsilon$ is relatively compact in $L^2(0, T; H^{s+1}_N(\Omega))$ also yields
$$\int_0^t [ u_\epsilon(r) ]^2_{s+1, \Omega, N} \, dr = \lim_{t \to 0} \int_0^t [ u^\tau_\epsilon(r) ]^2_{s+1, \Omega, N} \, dr. $$
Thus, we claim that
\begin{align*}
\int_{0}^{t} \int_{\Omega}S^\tau(r, x) G'_\epsilon(u^\tau_\epsilon)   \,dx dr \longrightarrow \int_{0}^{t} \int_{\Omega} S(r,x) G'_\epsilon(u_\epsilon) \,  dx dr.
\end{align*}
Indeed, a simple computation shows
\begin{align*}
&\int_{0}^{t} \int_{\Omega} S^\tau(r, x)G'_\epsilon(u^\tau_\epsilon)  - S(r, x) G'_\epsilon(u_\epsilon) \, dx dr\\
= &\int_{0}^{t} \int_{\Omega} S^\tau(r, x) \left(G'_\epsilon(u^\tau_\epsilon)- G'_\epsilon(u_\epsilon)\right)\, dx dr
+ \int_{0}^{t} \int_{\Omega} G'_\epsilon(u_\epsilon) \left(S^\tau(r, x)- S(r, x)\right)\, dx dr\\
\le& \int_{0}^{t} \int_{\Omega} \frac{1}{\epsilon} |u^\tau_\epsilon-u_\epsilon| S^\tau(r, x) \, dx dr
+\int_{0}^{t} \int_{\Omega} \frac{1}{\epsilon} |u_\epsilon-A| \left(S^\tau(r, x)- S(r, x)\right) \, dx dr\\
\le &\,\frac{1}{\epsilon} \|u^\tau_\epsilon-u_\epsilon\|_{L^\infty(0,T;L^1(\Omega))}\|S^\tau(r, x)\|_{L^1(0,T;L^1(\Omega))}\\
&\quad+ \frac{1}{\epsilon} \|u_\epsilon-A\|_{L^\infty(0,T;L^1(\Omega))}\|S^\tau(r, x)-S(r, x)\|_{L^1(0,T;L^1(\Omega))}
\longrightarrow 0, ~~\text{as~} \tau \to 0.
\end{align*}
Hence, \eqref{Stationary.entropy1} implies \eqref{exist.entropy}.  The proof of the Proposition \ref{exist.regular} is complete.

\subsection{Proof of Theorem \ref{ex1}}
In this subsection, we establish uniform a priori estimates for the regularized solution, independent of the parameter $\epsilon$. 
Below, we present a complete proof of Theorem \ref{ex1}, which is divided into seven steps.

\textbf{Step 1: A priori estimates.} Similar to the proof of \eqref{zxc3}, we can get that $f_\epsilon(u_\epsilon)$ is bounded in $L^\infty(0, T; L^p(\Omega))$ with $p < \frac{2}{1 - 2s}$, if $0 < s \le  \frac{1}{2}$; $p =\infty$, if $\frac{1}{2} < s <1$.
But the difference is that we can not get $\partial_x I(u_\epsilon)$ boundedness, we can only get $f_\epsilon(u_\epsilon)^{\frac{1}{2}} \partial_x I(u_\epsilon)$ is bounded in $L^2(0, T; L^2(\Omega))$ from \eqref{exist.energy}, which implies that $f_\epsilon(u_\epsilon) \partial_x I(u_\epsilon)$ is bounded in $L^2(0, T; L^r(\Omega))$ where $\frac{1}{r} = \frac{1}{2} + \frac{1}{2p}$. Therefore, we obtain that
$$
\partial_t u_\epsilon = -\partial_x \left( f_\epsilon(u_\epsilon) \partial_x I(u_\epsilon) \right) +S(t,x) \text{ is bounded in } L^2(0, T; W^{-1, r}(\Omega)),
$$	
which indicates that $ u_\epsilon $  satisfies the following estimate
\begin{align*}	
\left\|\partial_t u_\epsilon\right\|_{L^{2}(0,T;W^{-1,r}(\Omega))} \leq C,
\end{align*}
where $ C $ does not depend on $ \epsilon > 0 $  and  $r=\frac{2p}{p+1}$.

\textbf{Step 2: Compactness result.}
The embedding relationships \eqref{embedding} and Aubin's lemma indicate that $u_\epsilon$ is relatively compact in $C(0, T; L^p(\Omega))$ with $p < \frac{2}{1 - 2s}$ if $0 < s \le  \frac{1}{2}$   and $C(0, T; C^{\frac{2s-1}{2}}(\Omega))$  if $\frac{1}{2} < s <1$.
Therefore, we can extract a subsequence, also denoted $u_\epsilon$, such that when $\epsilon$ approaches zero, we have
$u \in L^\infty(0, T; H^s(\Omega)) \cap L^2(0, T; H^{s+1}_N(\Omega))$ satisfying
\begin{equation*}
\begin{cases}
u_\epsilon \to u \text{~in~} C(0, T; L^p(\Omega)) \text{~strongly~for all~} p < \frac{2}{1 - 2s},  ~&\text{~if~} 0 < s \le  \frac{1}{2},\\
u_\epsilon \to u \text{~in~} C(0, T; C^{\frac{2s-1}{2}}(\Omega))  \text{~strongly},  ~&\text{~if~} \frac{1}{2} < s <1.
\end{cases}
\end{equation*}
And then, on the one hand, if $0 < s \le  \frac{1}{2}$, then we have
$$u_\epsilon \to u  \text{~almost everywhere in~} \Omega_T.$$
On the other hand, if $\frac{1}{2} < s <1$, then Egorov's theorem implies that for any $\eta>0$, there exists a set $E_\eta \subset \Omega_T$ with $|\Omega_T \setminus E_\eta| \le \eta$ such that
\begin{equation}\label{local1}
u_\epsilon \to u \text{~uniformly~in~} E_\eta.
\end{equation}

\textbf{Step 3: Derivation of definition \eqref{ex1.defi}.}
Let us pass to the limit in \eqref{exist.defi}. Let $\varphi \in \mathcal{D}([0, T) \times \bar{\Omega})$ satisfying $\partial_x \varphi = 0$ on $(0, T) \times \partial \Omega$. Since $u_\epsilon \to u$ in $C(0, T; L^1(\Omega))$, we have
$$  \iint_{\Omega_T} u_\epsilon \partial_t \varphi \, dxdt \to \iint_{\Omega_T} u \partial_t \varphi \, dx dt.  $$
Remark that \eqref{exist.energy}, we know
\begin{align}\label{epsilon}
\iint_{\Omega_T}  f_\epsilon(u_\epsilon) &\left| \partial_{x} I(u_\epsilon)\right| ^{2} \, dx dt\nonumber\\
&= \iint_{\Omega_T}  (u_\epsilon)^n_+ \left| \partial_{x} I(u_\epsilon)\right| ^{2} \, dx dt
+\iint_{\Omega_T}  \epsilon \left| \partial_{x} I(u_\epsilon)\right| ^{2}  \,dx dt
\leq C.
\end{align}
The Cauchy-Schwarz inequality and \eqref{epsilon} indicate that

$$\epsilon \iint_{\Omega_T} \partial_x I(u_ \epsilon) \partial_x \varphi \, dx dt \le c(\varphi) \sqrt{ \epsilon} \left( \sqrt{ \epsilon} \|\partial_x I(u_\epsilon)\|_2 \right) \to 0.  $$
Estimation \eqref{epsilon} also shows that $ (u_\epsilon)^{\frac{n}{2}}_+\partial_{x} I(u_\epsilon)$ is bounded in $L^2(0, T; L^2(\Omega))$.

In the last of this step, we will estimate the term $(u_\epsilon)^{\frac{n}{2}}$.  Indeed,  for $0 < s \le  \frac{1}{2}$, since the function $u \mapsto u^{\frac{n}{2}}$ is Lipschitz for $n\ge2$, we easily verify
$(u_ \epsilon)^{\frac{n}{2}}$ is bounded in $L^\infty(0, T; L^p(\Omega))$ with $p < \frac{2}{1-2s}$. Thus, $ (u_\epsilon)_+^n \partial_x I(u_\epsilon)$ is bounded in $L^2(0, T; L^m(\Omega))$, where $\frac{1}{m} = \frac{1}{2} + \frac{1}{p}$.
Different from the Lipschitz case, when $n<2$, we only prove that $(u_\epsilon)^{\frac{n}{2}}$ is bounded in $L^\infty(0, T; L^{\frac{2p}{n}}(\Omega))$ with $p < \frac{2}{1-2s}$. In short, we conclude that $ (u_\epsilon)_+^n \partial_x I(u_\epsilon)$ is bounded in $L^2(0, T; L^m(\Omega))$, where $\frac{1}{m} = \frac{1}{2} + \frac{n}{2p}$.  For  $\frac{1}{2} < s <1$, it is easy to verify that $ (u_\epsilon)_+^n \partial_x I(u_\epsilon)$ is bounded in $L^2(0, T; L^2(\Omega)).$

\textbf{Step 4: Equation for the flux $h$.}
Taking the limit of  \eqref{exist.defi} yields
$$  \iint_{\Omega_T} u \partial_t \varphi \, dx dt + \iint_{\Omega_T} h \partial_x \varphi  \, dx dt = -\int_\Omega u_0 \varphi(0, x) \, dx.  $$
We further claim that
$$  h = u_{+}^n \partial_x I(u),  $$
in the following sense:
\begin{align}\label{h}
\iint_{\Omega_T} h \varphi \, dx dt
= - \iint_{\Omega_T} n  u_{+}^{n-1} \partial_x u I(u)  \varphi \, dx dt
-\iint_{\Omega_T} u_{+}^n  I(u) \partial_x \varphi \, dx dt,
\end{align}
for all test functions $\varphi$ with $\varphi = 0$ on $(0, T) \times \partial \Omega$. This is
$$  h = \partial_x (u_{+}^n I(u)) - nu_{+}^{n-1} \partial_x u I(u) \text{ in } \mathcal{D}'(\Omega).  $$
Indeed, from \eqref{exist.energy}, we have
\begin{align}\label{uhs}
[u_\epsilon]_{s,\Omega}
\le [u_{0}]_{s,\Omega}
+ \int_0^t[S(r,x)]_{s,\Omega}\,dr.
\end{align}
Additionally, using the fractional Poincar\'e's  inequality and the mass equation \eqref{exist.mass}, we obtain
\begin{align}\label{u.1}
\|u_\epsilon\|_2&\le \|u_\epsilon-\bar{u}_\epsilon\|_2+\|\bar{u}_\epsilon\|_2\nonumber\\
&\le  C(\Omega)[u_\epsilon]_{s,\Omega}+\|\bar{u}_0\|_2+|\Omega|^{-\frac{1}{2}}\int_0^t\int_\Omega S(r,x)\,dxdr.
\end{align}
Hence, for  $s \in (\frac{1}{2},1)$,  the embedding $H^s(\Omega) \hookrightarrow L^{\infty}(\Omega)$ together with  \eqref{uhs} and \eqref{u.1} yields
\begin{align}\label{u-u}
u_\epsilon(t,x)&\le \|u_\epsilon\|_{H^s(\Omega)}\nonumber\\
&\le C\|u_0\|_{H^s(\Omega)}+\int_0^t[S(r,x)]_{s,\Omega}dr+C\int_0^t\int_{\Omega} S(r, x)\,dxdr \le A<+\infty.
\end{align}
From the definition of $G_\epsilon(u_\epsilon)$, it follows that
\begin{align*}
G'_\epsilon(u_\epsilon) =\int^{u_\epsilon}_{A}\frac{1}{f_\epsilon(z)}dz \le 0, ~~~\forall u_\epsilon \le A.
\end{align*}
If $S(t,x)\ge 0$, then from \eqref{exist.entropy} we derive
\begin{align}\label{entropy.V}
\int_{\Omega} G_\epsilon(u_\epsilon)  \,dx + \int_0^{t} [u_\epsilon]_{s+1,\Omega,N}^{2}  \,dr
&\leq \int_{\Omega} G_\epsilon(u_0) \, dx
+\int_{0}^{t} \int_{\Omega} S(r,x) G'_\epsilon(u_\epsilon) \,  dx dr\nonumber\\
&\leq \int_{\Omega} G_\epsilon(u_0) \, dx.
\end{align}
Further, with the help of the monotonicity of  $G_\epsilon$ with respect to $\epsilon$, we have
$$
\int_\Omega G_\epsilon(u_0) \, dx \leq \int_\Omega G(u_0) \, dx \leq C.  $$
Hence,  \eqref{entropy.V}  implies that $u_\epsilon$ is bounded in $L^2(0, T; H^{s+1}_N(\Omega))$. Since $\partial_t u_\epsilon$ is bounded in $L^2(0, T; W^{-1, r}(\Omega))$, applying Aubin's lemma yields
\begin{equation*}
\begin{cases}
u_\epsilon \text{ is relatively compact in } L^2(0, T; H^{s_1}_N(\Omega)) \text{ with } s_1 < s+ 1,\\
\partial_x u_\epsilon \text{ is relatively compact in } L^2(0, T; H^{s_2}_N(\Omega)) \text{ with } s_2 < s,  \\
I(u_\epsilon) \text{ is relatively compact in } L^2(0, T; H^{s_3}_N(\Omega)) \text{ with } s_3 < 1 - s.
\end{cases}
\end{equation*}
Therefore,  we can extract a subsequence such that
if $\frac{1}{2} < s <1$, then Egorov's theorem shows that for any $\eta>0$,  there exists  sets $E_\eta,~A_\eta \subset \Omega_T$ with $|\Omega_T \setminus E_\eta| \le \eta$ and $|\Omega_T \setminus A_\eta| \le \eta$ such that
\begin{equation}\label{u.local}
\begin{cases}
u_\epsilon \to u \text{~uniformly~in~} E_\eta,\\
\partial_x u_\epsilon \to \partial_x u \text{~uniformly~in~}  A_\eta, \\
I(u_\epsilon) \to I(u) \text{ in }  L^2(0, T; L^{\frac{2}{2s+2\varepsilon-1}}(\Omega)) \text{ with } \varepsilon>0.
\end{cases}
\end{equation}
Further, we write
\begin{align*}
\iint_{\Omega_T} h_\epsilon \varphi \, dx dt &= \iint_{\Omega_T} (u_\epsilon)_{+}^{n} \partial_x I(u_\epsilon) \varphi \, dx dt\nonumber\\
&= -\iint_{\Omega_T} n(u_\epsilon)_{+}^{n-1} \partial_x u_\epsilon I(u_\epsilon) \varphi \, dx dt - \iint_{\Omega_T} (u_\epsilon)_{+}^{n} I(u_\epsilon) \partial_x \varphi \, dx dt.
\end{align*}
From above convergence,  we can take the limit and obtain \eqref{h}.

\textbf{Step 5: The properties of $u$.}
Regarding the properties of $u$, taking the limit in \eqref{exist.mass} indicates the mass equation \eqref{ex1.mass} by $u_\epsilon \to u$ in $C(0, T; L^1(\Omega))$.

Next,  noting that $u_\epsilon$ is bounded in $L^\infty(0, T; H^s(\Omega))$, then Fatou's lemma implies
$$[u]_{s,\Omega}\leq \liminf_{\epsilon \to 0} \,[u_\epsilon]_{s,\Omega}.  $$
Estimate \eqref{exist.energy} also implies that $$g_\epsilon := (u_\epsilon)_+^{\frac{n}{2}} \partial_x I(u_\epsilon) \rightharpoonup g \text{~in~} L^2(0, T; L^2(\Omega))  \text{~weakly} ,$$
and then the lower semi-continuity of the norm suggests \eqref{ex1.energy}.
Now, we only prove that
\begin{align}\label{gep}
g = \partial_x\left(u_+^{\frac{n}{2}}I(u) \right) - \frac{n}{2}u_{+}^{\frac{n}{2} - 1} \partial_x u I(u) \text{ in } \mathcal{D}'(\Omega).
\end{align}
Indeed, for all $\varphi \in \mathcal{D}([0, T ) \times \bar{\Omega} ) $, integration by parts shows
\begin{align*}
\iint_{\Omega_T}  g_\epsilon \varphi \, dx dt
&= \iint_{\Omega_T}  \left(u_\epsilon\right)_+^{\frac{n}{2}} \partial_x I\left(u_\epsilon\right) \varphi \, dxdt\nonumber\\
&= -\iint_{\Omega_T}  \frac{n}{2} \left(u_\epsilon\right)_+^{\frac{n}{2}-1} \partial_x u_\epsilon I\left(u_\epsilon\right) \varphi \, dxdt
- \iint_{\Omega_T}  \left(u_\epsilon\right)_+^{\frac{n}{2}} I\left(u_\epsilon\right) \partial_x \varphi \, dx dt.
\end{align*}
Hence, using  \eqref{u.local} and $u_\epsilon \rightarrow u$ in $C(0, T; L^\infty(\Omega))$, we have \eqref{gep}.

Furthermore, we show that $u$ satisfies the entropy inequality \eqref{ex1.entropy}.
First of all, since $u_\epsilon \rightharpoonup u$ in $L^2(0, T; H^{s + 1}_N(\Omega))$,  by the lower semi-continuity of the norm, we obtain
$$
\|u\|_{L^2(0, T; H^{s + 1}_N(\Omega))} \leq \liminf_{\epsilon \to 0} \|u_\epsilon\|_{L^2(0, T; H^{s + 1}_N(\Omega))}.
$$
According to the facts that
$$ G_\epsilon (u_\epsilon) \to G(u) \text{ almost everywhere and } G_\epsilon(u_0) \leq G(u_0). $$
Then estimate \eqref{ex1.entropy} follows from \eqref{exist.entropy} by Fatou's lemma.

\textbf{Step 6: The non-negative  of  $u$.}
The non-negativity of the solution follows from the entropy estimate \eqref{entropy.V}. Indeed,  it follows from \eqref{entropy.V} that
\begin{align}\label{non1}
\int_{\Omega} G_{\epsilon}(u_{\epsilon})dx \leq \int_{\Omega} G_{\epsilon}(u_{0})dx \leq \int_{\Omega} G(u_{0})dx   < \infty,
\end{align}
which implies
\begin{align}\label{non2}
\limsup_{\epsilon \to 0} \int_{\Omega} G_{\epsilon}(u_{\epsilon})dx < \infty.
\end{align}
From the definition of $G_{\epsilon}(z)$, we know
\[
\lim_{\epsilon \to 0} G_{\epsilon}(-\xi) = +\infty,   ~~~~\forall \xi > 0.
\]
Recall that $u_{\epsilon}(t, \cdot)$ converges almost everywhere, 
then Egorov's theorem implies that  for any  $\eta>0$, there exists a set $E_{\eta} \subseteq \Omega$ with $|\Omega \setminus E_{\eta}| \leq \eta$ such that
\[
u_{\epsilon}(t, \cdot) \to u(t, \cdot) \text{ uniformly in } E_{\eta}.
\]
For some $\xi > 0$, we let
\[
Q_{\eta, \xi} = E_{\eta} \cap \{u(t, \cdot) \leq -2\xi\}.
\]
Then for every $\eta, \xi > 0$, there exists $\epsilon_{0}(\eta, \xi)$ such that 
$u_{\epsilon}(t, \cdot)\le -\eta$ in $Q_{\eta, \xi}$  for all $\epsilon\le \epsilon_{0}(\eta, \xi)$.

Next, we claim  that $Q_{\eta, \xi}$ has measure zero. Indeed, if not, then there exists $\epsilon_1\le \epsilon_{0}(\eta, \xi)$ such that
$u_{\epsilon_1}(t, \cdot)> -\eta$ for all $x\in Q_{\eta, \xi}.$
Further, we get
\[
G_{\epsilon_1}(u_{\epsilon_1}) \geq G_{\epsilon_1}(-\xi) \to +\infty,  ~~~\forall x\in Q_{\eta, \xi}.
\]
Therefore,  Fatou's lemma implies
\[
\liminf_{\epsilon_1 \to 0} \int_{Q_{\eta,\xi}} G_{\epsilon_1}(u_{\epsilon_1})dx \geq \int_{Q_{\eta,\xi}} \liminf_{\epsilon_1 \to 0} G_{\epsilon_1}(u_{\epsilon_1})dx = +\infty,
\]
which contradicts \eqref{non2}.

Hence, for all $\xi > 0$, $\eta > 0$, we know
\[
|\{u(t, \cdot) \leq -2\xi\}| \leq |Q_{\eta,\xi}| + |\Omega \setminus E_{\eta}| \leq \eta.
\]
It follows that
$$|\{u(t, \cdot) \leq -2\xi\}| = 0, ~~~\forall \xi > 0.$$
Furthermore, we obtain
\[
|\{u(t, \cdot) < 0\}| = \left|\bigcup_{k=1}^{\infty} \left\{u(t, \cdot) \leq -\frac{1}{k}\right\}\right|=0.
\]
Thus, we deduce  $u(t,x) \geq 0$ for almost every $x \in \Omega$ and for all $t > 0$.

\textbf{Step 7: The long-time behavior  of weak solutions.}
In this step, we will prove that $u$ satisfies the estimate \eqref{ex3.uK1}.
First from  \eqref{exist.energy}, we have
\begin{align}\label{exist.energy1}
\frac{d}{dt}[u_\epsilon]^2_{s,\Omega}
+2 \int_{\Omega} f_\epsilon(u_\epsilon) \left| \partial_{x} I(u_\epsilon)\right| ^{2}  dx
\le 2[S(t,x)]_{s,\Omega}  [u_\epsilon]_{s,\Omega}.
\end{align}
According to \eqref{non1}, we know
\begin{align}\label{xxx1}
\int_\Omega \frac{\left( u_\epsilon\right) ^2}{f_\epsilon(u_\epsilon) } \, dx
\le \int_\Omega G_\epsilon(u_\epsilon)\,dx \le \int_\Omega G(u_0)\,dx :=H_0< +\infty,
\end{align}
where $H_0$ is independent of $t$. Together with Lemma \ref{2007lem},  we further get
\begin{align}\label{xxx2}
\int_\Omega f_\epsilon(u_\epsilon)  |\partial_x (-\Delta)^s u_\epsilon|^2 \, dx
\geq  \frac{1}{4|\Omega|^2 H_0}\left( \int_\Omega | (-\Delta)^{\frac{s}{2}} u_\epsilon|^2 \,dx \right) ^2.
\end{align}	

Let us denote by
$$J(u_\epsilon) := \int_{\Omega} |(-\Delta)^{\frac{s}{2}} u_\epsilon|^2 \,dx.  $$
Then from \eqref{exist.energy1}, we deduce
\begin{align}\label{xxx3}
\frac{d}{dt} J(u_\epsilon)+ \frac{1}{2|\Omega|^2H_0} J^2(u_\epsilon)
\le 2[S(t,x)]_{s,\Omega}  J^{\frac{1}{2}}(u_\epsilon).
\end{align}
Multiplying both sides of \eqref{xxx3} by $J^{-\frac{1}{2}}(u_\epsilon)$, we get
\begin{align}\label{dissipation.zz}
\frac{d}{{d}t}[u_\epsilon]_{s,\Omega} 
+ \frac{1}{4|\Omega|^2H_0}[u_\epsilon]_{s,\Omega}^3
\le [S(t,x)]_{s,\Omega},~~~\forall t > 0.
\end{align}
According to \eqref{dissipation.zz} and Lemma \ref{inequ}, we take the limit as $\epsilon \rightarrow 0$ to get 
\begin{align}\label{E.1}
[u_0]_{s,\Omega} & \left( 1+ \frac{1}{2|\Omega|^2H_0} [u_0]_{s,\Omega}^2t \right)^{-\frac{1}{2}} \le [u]_{s,\Omega}\nonumber\\
&~~~~\le  M_0\left( 1+\frac{(1-\kappa)M_0^2}{2|\Omega|^2H_0} t \right) ^{-\frac{1}{2}}
+ \int_{\kappa t}^t [S(r,x)]_{s,\Omega}\, dr,~~~\forall t>0,
\end{align}
where $\kappa\in(0,1),$ $0<M_0:=[u_0]_{s,\Omega} + \int_0^\infty [S(r,x)]_{s,\Omega}\, dr <+\infty.$ 
Thus, from fractional Poincar\'e's  inequality and \eqref{E.1}, we have
\begin{align}\label{H1.1}
[u]_{s,\Omega} 
\le 	&\left\| u - \bar{u}_0 - \frac{1}{|\Omega|} \int_{0}^\infty \int_{\Omega}  S(r,x) \, dx dr\right\|_{H^s(\Omega)} \nonumber\\
\le & \left\| u - \bar{u}_0 - \frac{1}{|\Omega|} \int_{0}^t \int_{\Omega}  S(r,x) \, dx  dr\right\|_{H^s(\Omega)} +\left\|\frac{1}{|\Omega|} \int_{t}^\infty \int_{\Omega}  S(r,x) \, dx  dr  \right\|_{H^s(\Omega)}\notag\\
\le &\left\| u - \bar{u}_0 - \frac{1}{|\Omega|} \int_{0}^t \int_{\Omega}  S(r,x) \, dx  dr  \right\|_2
+ [u]_{s,\Omega}+  \left\|\frac{1}{|\Omega|} \int_{t}^\infty \int_{\Omega}  S(r,x) \, dx  dr  \right\|_2 \notag\\
\leq & \; C M_0\left( 1+\frac{(1-\kappa)M_0^2}{2|\Omega|^2H_0} t \right) ^{-\frac{1}{2}}
+ C \int_{\kappa t}^t [S(r,x)]_{s,\Omega}\, dr \notag\\
&+ |\Omega|^{-\frac{1}{2}}\int_{t}^\infty \int_{\Omega}  S(r,x) \, dx  dr ,    ~~~\forall t>0.
\end{align}
Combining \eqref{E.1} and \eqref{H1.1} yields \eqref{ex3.uK1}.  
Then the proof of the Theorem \ref{ex1} is complete.
$\square$


As a byproduct, we derive the following local  $L^1$-in-time decay estimate for the dissipation functional.
\begin{theorem}\label{D}
	Assume that all conditions of Theorem \ref{ex1} hold, and let $u$ be a non-negative weak solution. Then there exists a constant $C>0$, independent of $t$, such that the dissipation functional $D(u)(t)$  satisfies	
	\begin{align}
	\int_{\frac{t}{2}}^{t}D(u)(r) dr\leq  \frac{C}{t}\int_{\frac{t}{4}}^{\frac{t}{2}}J(u)(r)d r+C\left( \int_{\frac{t}{4}}^{t} [S(r,x)]_{s,\Omega}^2 d r\right)^{\frac{1}{2}}\left( \int_{\frac{t}{4}}^{t}J(u)(r) dr\right)^{\frac{1}{2}},
	\end{align}
	where $D(u)(t) = \int_{\Omega} u^{n} |\partial_{x}(-\Delta)^{s}u|^{2}dx$.
\end{theorem}

\begin{proof} We first choose a cut-off function $\eta \in C^{\infty}(\mathbb{R})$ satisfying 
	$0\leq \eta\leq1,\eta(r)=1$ for $r\geq\frac{t}{2}$, $\eta(r)=0$ for $r\leq\frac{t}{4}$ and $\eta'(r)\leq\frac{C}{t}$ for some constant $C>0.$
	Taking $-\eta (-\Delta)^{s}u$ as a test function, we derive
	\begin{align}\label{add01}
	\int_{0}^{t}\int_{\Omega}\eta(r)u_{r}(-\Delta)^{s}u\,d xd r+\int_{0}^{t} \int_{\Omega} &\eta(r)u^{n} |\partial_{x}(-\Delta)^{s}u|^{2}\,d xd r\nonumber\\
	&=\int_{0}^{t}\int_{\Omega} \eta(r)(-\Delta)^{\frac{s}{2}}S(r,x)(-\Delta)^{\frac{s}{2}}u \,d xd r.
	\end{align}
	Noticing that $\eta(0)=0$ and $\eta(t)=1$, we use the definition of $J(u)(t)$ to compute
	\begin{align}\label{add02}
	0\leq J(u)(t)=\int_{0}^{t}\frac{d }{dr}(\eta(r)J(u)(r))d r
	=\int_{0}^{t}\eta'(r)J(u)(r)d r+2\int_{0}^{t}\int_{\Omega}\eta(r)u_{r}(-\Delta)^{s}u\,d xd r.
	\end{align}		
	Combining \eqref{add01} with \eqref{add02} yields
	\begin{align}\label{add03}
	\int_{\frac{t}{2}}^{t}D(u)(r)d r&\leq \int_{0}^{t}\eta(r)D(u)(r)d r
	\nonumber\\ 
	&\leq \frac{1}{2}\int_{0}^{t}\eta'(r)J(u)(r)d r+\int_{0}^{t}\int_{\Omega} \eta(r)(-\Delta)^{\frac{s}{2}}S(r,x)(-\Delta)^{\frac{s}{2}}u\,d xd r
	\nonumber\\ 
	&\leq \frac{C}{t}\int_{\frac{t}{4}}^{\frac{t}{2}}J(u)(r)d r+C\left( \int_{\frac{t}{4}}^{t} [S(r,x)]_{s,\Omega}^2 d r\right)^{\frac{1}{2}}\left( \int_{\frac{t}{4}}^{t}J(u)(r) d r\right)^{\frac{1}{2}}.
	\end{align}
	Thus, we complete the proof of Theorem \ref {D}.
\end{proof}

\section{Time-and space-independent forces  $S_0$}
In this section, for time-and space-independent forces  $S_0$ case, we obtain stronger results for the long-time behavior of solutions. 
Here, we only consider $S_0 \ge 0$; otherwise, it follows from \eqref{ex1.mass} that
$$u(t, x) \rightarrow 0, \text{ as } t \rightarrow T^* := -\frac{\bar{u}_0 }{S_0},  $$
which means that the thin-film will completely dry out over the finite time  $T^*$.

{\bf Proof of Theorem \ref{ex2}.} Let
$$w(t, x) := u(t, x) - u_\Omega(t), $$
where
\begin{align}\label{uOmega}
u_{\Omega}(t) = \frac{1}{|\Omega|} \int_\Omega u(t,x)\,dx= \bar{u}_0  + tS_0.
\end{align}
Note that  $w = w_\epsilon(t, x)$  solves
\begin{equation}\label{con}
\begin{cases}
w_t + (f_\epsilon(u) \partial_xI(w))_x = 0,~~&(t,x) \in (0,T)\times \Omega,\\
\partial_x w =f_\epsilon(u) \partial_xI(w) = 0,~~&(t,x) \in (0,T) \times\partial \Omega,\\
w(0, x)=w_{0}(x):=u_{0}(x)-\bar{u}_0 ,~~&x \in \Omega.
\end{cases}
\end{equation}
Also, we have
\begin{align*}
\int_{\Omega} w(t, x) \,dx = 0.
\end{align*}
Multiplying \eqref{con} by $ (-\Delta)^s w$  and integrating in $ \Omega $, we obtain
\begin{align}\label{con.equ}
\frac{1}{2} \frac{d}{dt} \int_{\Omega} \left| (-\Delta)^{\frac{s}{2}}w\right| ^2  dx + \int_{\Omega} f_\epsilon(u) \left|\partial_x (-\Delta)^{s}w\right|^2  dx = 0.
\end{align}
The nonnegativity of the second  term on the left-hand side of \eqref{con.equ} indicates
$$[ w ]_{s,\Omega}  \le [w_0 ]_{s,\Omega},  $$
then by embedding theorem and fractional Poincar\'e's  inequality, we know
\begin{align}\label{1111}
\left|u- \bar{u}_0   - tS_0\right|
\le C(\Omega)[u]_{s,\Omega}
\le C(\Omega) [u_{0}]_{s,\Omega},
\end{align}
which implies
\begin{align}\label{con.v}
u(t, x) \ge v(t) := \bar{u}_0  +  tS_0 - C(\Omega)[u_{0}]_{s,\Omega}> 0,
\end{align}
for all
\begin{equation}\label{T0}
t \geq T_0 := \frac{\left[C(\Omega)[u_{0}]_{s,\Omega}- \bar{u}_0  \right] _+}{S_0}.
\end{equation}	
Due to the interpolation inequality, Poincar\'e's  inequality and Young's inequality, we find
\begin{align*}
\|(-\Delta)^{\frac{s}{2}} w\|_2
&\le \|(-\Delta)^{s+\frac{1}{2}} w\|_2^{\frac{s}{2s+1}}  \|w\|_2^{\frac{s+1}{2s+1}}
\le C\|(-\Delta)^{s+\frac{1}{2}} w\|_2^{\frac{s}{2s+1}}  \|(-\Delta)^{\frac{s}{2}} w\|_2^{\frac{s+1}{2s+1}}\nonumber\\
&\le \varepsilon\|(-\Delta)^{\frac{s}{2}} w\|_2
+C\|(-\Delta)^{s+\frac{1}{2}} w\|_2,
\end{align*}
which implies
\begin{align}\label{con.inter}
\|(-\Delta)^{\frac{s}{2}} w\|_2^2
\le C\|(-\Delta)^{s+\frac{1}{2}} w\|_2^2 = C\int_{\Omega}\left|\partial_{x} (-\Delta)^{s} w \right|^2 dx.
\end{align}
Substituting \eqref{con.v} and \eqref{con.inter} into \eqref{con.equ} yields
\begin{align}\label{con.ineq}
\frac{d}{dt} \int_{\Omega} \left| (-\Delta)^{\frac{s}{2}}w\right| ^2  dx
+ Cf_\epsilon(v(t))\int_{\Omega}  \left| (-\Delta)^{\frac{s}{2}}w\right|^2  dx \le 0,  \quad \forall t \ge T_0.
\end{align}   	
Estimates \eqref{1111} and \eqref{con.ineq} shows
\begin{align}\label{con.ineq1}
[ w ]_{s,\Omega}^2
\le[ w(T_0) ]_{s,\Omega}^2 e^{-\int_{T_0}^{t} Cf_\epsilon(v(r) \, dr}
\le [w_{0} ]_{s,\Omega}^2 e^{-\int_{T_0}^{t} Cf_\epsilon(v(r))\,dr}, \quad \forall t > T_0.
\end{align}
Then by \eqref{con.ineq1} and the fractional Poincar\'e's  inequality, we deduce
\begin{align}\label{con.ineq2}
\| w \|_{H^s(\Omega)}^2 \le C [ w_{0} ]_{s,\Omega}^2 e^{-C\int_{T_0}^{t} f_\epsilon(v(r)) \, dr}, \quad \forall t \geq T_0.
\end{align}
Hence, letting  $\epsilon \rightarrow 0$  in \eqref{con.ineq2}, we have
\begin{align*}
\| u-\bar{u}_0-tS_0 \|_{H^s(\Omega)}^2
\le C \| u_{0} \|_{{H}^s(\Omega)}^2 e^{-C\int_{T_0}^{t} \left(\bar{u}_0  - C(\Omega)[u_{0}]_{s,\Omega} +  S_0r \right) ^n  \, dr}, \quad \forall t \geq T_0.
\end{align*}
The proof of the Theorem \ref{ex2} is complete.
\(\square\)

\section{Comments and further problems}

The main contribution of this paper lies in the first derivation of bilateral estimates for the convergence rate of solutions to the nonlocal thin-film equation, and under specific conditions, this convergence rate is optimal.

{\bf Further problems.} Several important questions naturally arise from this study and deserve further investigation:

$\bullet$ Do analogous long-time  behaviors persist when $N = 1$ and $s \in (0, \frac{1}{2}]$, or when $N \ge 2$?  However, for $s \in (0, \frac{1}{2}]$, the lack of the embedding $H^s(\Omega) \hookrightarrow L^{\infty}(\Omega)$ poses a technical obstacle to establishing estimate \eqref{u-u}. Resolving these issues therefore requires novel methods.

$\bullet$ How would the solutions behave under more complex external forces,  e.g.   periodic, impulsive, or stochastic forcing?

$\bullet$ How to establish the interface propagation properties (e.g., finite speed propagation and waiting time phenomenon) of the nonlocal thin film equation with inhomogeneous forces?

~

\section*{Appendix A}


First, we denote $V = H^{2s+1}_N(\Omega)$. It is easy to know that the functional $\Gamma(u)$ is linear on $V$ for any $u \in V$,  and then using the embedding  $V \hookrightarrow L^\infty(\Omega)  $ is continuously embedded in $L^\infty(\Omega)$ and Proposition \ref{forall}, we obtain
$$
|\Gamma(u)(v)| \leq \left[ \|u\|_{H^s(\Omega)} + \tau  \|h\|_\infty\|u\|_{V} \right] \|v\|_{V}.
$$
Thus, The nonlinear operator $\Gamma$ is bounded.

Next, we will prove that the operator $\Gamma$ is coercive. From Proposition  \ref{forall}, we observe
\begin{align*}
\Gamma(u)(u) &\geq -\int_{\Omega} u I(u) \, dx + \left( \int_{\Omega} u \, dx \right)^{2} + c \tau \int_{\Omega} |\partial_x I(u)|^{2} \, dx  \nonumber\\
&\geq \min(1, \tau c) \|u\|^{2}_{H^{2s+1}_{N}(\Omega)}.
\end{align*}
We further deduce
$$
\frac{\Gamma(u)(u)}{\|u\|_V} \to +\infty, \quad \text{as~~}  \|u\|_V \to +\infty.
$$
Therefore, the operator $\Gamma$ is coercive.

Later, we only need to prove that $\Gamma$  is a monotone operator, namely the following Lemma \ref{pse}.
\begin{lemma}\label{pse}
	($\Gamma$ is pseudo-monotone.) Let $u_i $ be a sequence of functions in $ V $ satisfying $ u_i \rightharpoonup u $ weakly in $ V$. Then   	
	$$ \liminf_{i \to +\infty} \Gamma(u_i)(u_i - v) \geq \Gamma(u)(u - v).$$
\end{lemma}

\begin{proof}	
	First, with the aid of Proposition \ref{forall}, we have
	\begin{align}\label{Au}
	\Gamma(u_i)(u_i - v)=  &-\int_\Omega u_iI(u_i - v) \, dx + \frac{1}{|\Omega|}\left( \int_\Omega u_i \, dx \right)\left( \int_\Omega (u_i - v) \, dx \right)  \nonumber\\
	&+ \tau \int_\Omega h(u_i) \partial_x I(u_i) \partial_xI(u_i - v) \, dx  \nonumber\\
	=&\; \|u_i\|^2_{H^s(\Omega)} - \langle u_i, v \rangle_{H^s(\Omega)}  \nonumber\\
	&+ \tau \int_\Omega h(u_i)(\partial_xI(u_i))^2 \, dx - \tau\int_\Omega h(u_i) \partial_x(I(u_i)) \partial_x(Iv),  	
	\end{align}
	where
	$$\langle u, v \rangle_{H^s(\Omega)} = -\int_\Omega uI(v) \, dx +\left( \int_\Omega u\, dx \right)\left( \int_\Omega v \, dx \right) .$$
	Below, we need to ensure convergence in each of these terms.
	Since $ u_i $ converges weakly in $ V $, we can immediately obtain
	$$
	\liminf_{i \to +\infty} \|u_i\|^2_{H^s(\Omega)} \geq \|u\|^2_{H^s(\Omega)},
	$$
	and
	$$
	\lim_{i \to +\infty} \langle u_i, v \rangle_{H^s(\Omega)} = -\langle u, v \rangle_{H^s(\Omega)}.
	$$
	Then, we estimate the third term in the right-hand side of \eqref{Au}
	\begin{align*}
	\int_\Omega h(u_i)(\partial_xI(u_i))^2 \, dx
	&\geq -\|h(u_i) - h(u)\|_{\infty} \|u_i\|_{V}^2 + \int_\Omega h(u) (\partial_xI(u_i))^2 \, dx.
	\end{align*}
	It is not hard to find  that the first term of above inequality approaches zero according to the embedding  $V \hookrightarrow L^\infty(\Omega)  $,  we further obtain by using the lower semicontinuity of the $ L^2 $-norm
	$$
	\lim_{j \to \infty} \int_\Omega h(u_i)(\partial_xI(u_i))^2\geq \int_\Omega h(u) (\partial_xI(u))^2.
	$$
	Finally, we have
	$$
	h(u_i) \partial_x I(v) \to h(u) \partial_x I(v) \, \text{ in } L^2(\Omega) \text{ strongly,}
	$$
	$$
	\partial_x I(u_i) \to \partial_x I(u) \, \text{ in } L^2(\Omega) \text{ weakly,}
	$$
	which implies the convergence of the last term in \eqref{Au} and completes the proof.
\end{proof}

~

~

\noindent {\bf Funding} Bin Guo was partially supported by the National Natural Science Foundation of China (NSFC) (No. 11301211) and the Natural Science Foundation of Jilin Province, China (No. 201500520056JH).

~

\noindent {\bf Data Availability} No datasets were generated or analyzed during the current study.

~

\noindent {\bf Author Contributions} J.H Zhao and B. Guo wrote the main manuscript text and revised the manuscript. All authors reviewed the manuscript.

~

\section*{Declarations}

\noindent {\bf Conflict of interest} The authors declare that they have no conflict of interest.



          \end{document}